%% file: eb.tex
\documentclass[11pt]{article} % use larger type; default would be 10pt

\usepackage[utf8]{inputenc} % set input encoding (not needed with XeLaTeX)

\usepackage{cancel}

\usepackage{geometry} % to change the page dimensions
\usepackage{graphicx} % support the \includegraphics command and options
\usepackage[english]{babel}
\usepackage{booktabs} % for much better looking tables
\usepackage{array} % for better arrays (eg matrices) in maths
\usepackage{paralist} % very flexible & customisable lists (eg. enumerate/itemize, etc.)
\usepackage{verbatim} % adds environment for commenting out blocks of text & for better verbatim
\usepackage{subfig} % make it possible to include more than one captioned figure/table in a single float

\usepackage{amssymb}
\usepackage{bbm}
\usepackage{amsmath}
\usepackage{amsthm}
\usepackage{makecell}

\usepackage{epigraph}

\usepackage{hyperref}
\hypersetup{
    colorlinks=true,
    citecolor=blue
}

\usepackage{pdflscape}
\usepackage{natbib}
\usepackage{algorithm2e}

\usepackage{wrapfig}
\usepackage{subfig}
\usepackage[capitalize,noabbrev]{cleveref}
\usepackage{authblk}
\usepackage{changepage}

\usepackage{tikz-cd}
\usepackage{xcolor}

\input{macros}

\title{Sharp Matrix Empirical Bernstein Inequalities\footnote{NeurIPS 2025 Poster}}
\author[1]{Hongjian Wang}
\author[2]{Aaditya Ramdas}
\affil[1, 2]{Department of Statistics and Data Science, Carnegie Mellon University}
\affil[2]{Machine Learning Department, Carnegie Mellon University} 
\affil[ ]{\texttt{ \{hjnwang,aramdas\}@cmu.edu  }}

\date{\today}

\newtheorem{theorem}{Theorem}[section]

\newtheorem{proposition}[theorem]{Proposition}

\newtheorem{corollary}[theorem]{Corollary}
\newtheorem{lemma}[theorem]{Lemma}
\theoremstyle{definition}
\newtheorem{example}[theorem]{Example}

\Crefname{fact}{Fact}{Facts}

\begin{document}

\maketitle

\begin{abstract}
    \input{eb/abstract}
\end{abstract}

%{\tableofcontents}

\input{eb/maintext}

\subsubsection*{Acknowledgement}

\input{eb/ack}

\bibliography{urmat}

\newpage
\appendix

\input{eb/appdx}

%\section{Omitted Proofs}\label{sec:pf}

%\input{urmat/appendices}
\end{document}

%% file: macros.tex
\renewcommand{\le}{\leqslant}

\renewcommand{\geq}{\geqslant}
\renewcommand{\ge}{\geqslant}

\newcommand{\matle}{\preceq}

\newcommand{\matls}{\prec}

\newcommand {\trsp}{^\mathsf {T}}
\newcommand {\matl}{\left[ \begin{matrix}}
\newcommand {\matr}{\end{matrix}\right]}
\newcommand {\Exp}{ \mathbb E }

\renewcommand {\Pr}{ \mathbb P }

\newcommand{\mtx}[1]{\mathbf{#1}}
\newcommand{\vct}[1]{\mathbf{#1}}

\newcommand{\mX}{ \mtx{X} }
\newcommand{\mY}{ \mtx{Y} }
\newcommand{\mZ}{ \mtx{Z} }
\newcommand{\mV}{ \mtx{V} }
\newcommand{\mI}{ \mtx{I} }
\newcommand{\mM}{ \mtx{M} }
\newcommand{\mH}{ \mtx{H} }
\newcommand{\mU}{ \mtx{U} }

\newcommand{\mSg}{ \mtx{\Sigma} }
\newcommand{\mS}{ \mtx{S} }
\newcommand{\mA}{ \mtx{A} }
\newcommand{\mB}{ \mtx{B} }
\newcommand{\mC}{ \mtx{C} }

\newcommand{\mE}{ \mtx{E} }
\newcommand{\mF}{ \mtx{F} }
\newcommand{\mQ}{ \mtx{Q} }

\newcommand{\vv}{\vct{v}}

\newcommand {\Var}{\mathbf{Var}}

\newcommand{\e}{\mathrm e}

\newcommand{\cF}{\mathcal{F}}
\newcommand{\cS}{\mathcal{S}}

\newcommand{\psiE}{\psi_{\mathrm{E}}}

\DeclareMathAlphabet{\mathbbmsl}{U}{bbm}{m}{sl}

\newcommand{\hwcmt}[1]{}

\newcommand{\tr}{\mathrm{tr}}

%% file: eb/abstract.tex
We present two sharp, closed-form empirical Bernstein inequalities for symmetric random matrices with bounded eigenvalues. By sharp, we mean that both inequalities adapt to the unknown variance in a tight manner: the deviation captured by the first-order $1/\sqrt{n}$ term asymptotically matches the matrix Bernstein inequality exactly, including constants, the latter requiring knowledge of the variance. 
Our first inequality holds for the sample mean of independent matrices, and our second inequality holds for a mean estimator under martingale dependence at stopping times. 
%We also present generalizations of our bound in the direction of randomization and anytime-validity, as well as a handful of other matrix concentration bounds obtained via the same method. 

%% file: eb/maintext.tex
\section{Introduction}\label{sec:intro}

We are interested in nonasymptotic confidence sets for the common mean of independent or martingale-dependent bounded random matrices that optimally adapt to the unknown underlying variance. We first review the scalar case to set some context. 
%The familiar reader can skip directly to \Cref{sec:mat-eb}, where our main results are crisply presented as Propositions~\ref{prop:mpinformal} and~\ref{prop:informal}.

\subsection{Background: Scalar Empirical Bernstein Inequalities}

The classical Bennett-Bernstein inequality (see Lemma 5 of \cite{audibert2009exploration}; also \Cref{sec:rmk-bb}) states that, for the average $\overline X_n$ of independent random scalars $X_1,\dots, X_n$ with common expected value $\mu = \Exp X_i$, common almost sure upper bound $|X_i| \le B$, and second moment upper bound $ \sum_{i=1}^n \Exp X_i^2 \le n\sigma^2$,
\begin{equation}\label{eqn:bi}
    \Pr\left(  \overline X_n  - \mu \ge \frac{ B\log(1/\alpha)  }{3n} + \sqrt{\frac{2 \sigma^2 \log(1/\alpha)}{n}}    \right) \le \alpha.
\end{equation}
It is clear that \eqref{eqn:bi} remains true if the assumptions are centered instead: $ |X_i - \mu| \le B$ and $ \sum_{i=1}^n \Var(X_i) \le n \sigma^2$. A crucial feature of \eqref{eqn:bi} is that if $\sigma^2 \approx  \Exp X_1^2  \ll B^2$, the deviation is dominated by the ``variance term" $\Theta\left(\sqrt{n^{-1} \sigma^2 \log(1/\alpha)}\right)$, tighter than the ``boundedness term'' $\Theta(\sqrt{n^{-1} B^2 \log(1/\alpha) })$ that dominates if \citeauthor{hoeffding1963probability}'s inequality~\citeyearpar{hoeffding1963probability} is applied instead in the absence of the variance bound $\sigma^2$.

In practice, one often knows $B$ but has no prior access to the possibly much smaller $\sigma$.
Thus, such bounds are usually only used in theoretical analysis, but not to practically construct confidence bounds for the mean. For the latter task, so-called nonasymptotic \emph{empirical} Bernstein (EB) inequalities are therefore of particular interest. These inequalities often only assume the almost sure upper bound $B$ of the random variables and are agnostic and \emph{adaptive} to the true variances $\Var(X_i)$, to the effect that the final deviation is still dominated by an asymptotically $\Theta\left(\sqrt{n^{-1} \sigma^2 \log(1/\alpha)}\right)$ variance term, instead of $\Theta(\sqrt{n^{-1} B^2 \log(1/\alpha) })$ from Hoeffding's inequality under the \emph{same} boundedness assumption.

Scalar EB inequalities are derived from two very different types of techniques. First, a {union bound} between a non-empirical (``oracle'') Bernstein inequality and a concentration inequality on the sample variance, which is employed by early empirical Bernstein results \citep{audibert2009exploration,maurer2009empirical}. For example, for i.i.d., $[0,1]$-bounded $X_1,\dots,X_n$, and their Bessel-corrected sample variance $\hat \sigma^2_n$,  \citet[Theorem 4]{maurer2009empirical} prove the EB inequality
\begin{equation}\label{eqn:mp-eb}
     \Pr\left(  \overline X_n  - \mu  \ge  \sqrt{\frac{2 \hat \sigma^2_n \log(2/\alpha)}{n}}  +  \frac{ 7\log(2/\alpha)  }{3(n-1)}  \right) \le \alpha.
\end{equation}
Second, the self-normalization martingale techniques of \citet[Theorem 4]{howard2021time} and  \cite{waudby2020estimating},  which enable sharper rates, stopping time-valid concentration, martingale dependence, and variance proxy by predictable estimates other than the sample variance.  For example, \citet[Theorem 2, Remark 1]{waudby2020estimating} prove the following EB inequality for $[0,1]$-bounded random variables $X_1,\dots, X_n$ with common conditional mean $\mu=\Exp[X_i | X_1, \dots, X_{i-1}]$:
\begin{equation}\label{eqn:eb-informal}
    \Pr\left( \hat\mu_n -  \mu  \ge \sqrt{\frac{2 \log(1/\alpha) V_{n,\alpha}} {n}}   \right) \le \alpha.
\end{equation}
Above, $ \hat\mu_n $ is a particular weighted average of $X_1,\dots, X_n$, and $V_{n,\alpha} = V(\alpha, X_1,\dots, X_n)$ a quantity depending on the sample and $\alpha$ that converges to $\sigma^2$ almost surely
in the $n\to\infty$ limit, if $X_1,\dots, X_n$ have a common conditional variance $\sigma^2$.
\footnote{This was originally proved by \cite{waudby2020estimating} under i.i.d.\ assumption. We prove it under martingale dependence in our matrix result.}

These exact terms will become clear when we present our matrix result later in \Cref{sec:res} (taking $d=1$), but one can already notice the important fact that \eqref{eqn:eb-informal} matches \eqref{eqn:bi} asymptotically without requiring a known variance bound: letting $D_n=\sqrt{2n^{-1} \log(1/\alpha) V_{n,\alpha}}$ be the deviation term of \eqref{eqn:eb-informal}, 
\begin{equation}\label{eqn:sharp}
    \sqrt{n} D_n \stackrel{a.s.}\to \sqrt{2 \log(1/\alpha) \sigma^2 },
\end{equation}
a limit also attained by the oracle Bennett-Bernstein inequality \eqref{eqn:bi}'s deviation term. A one-sided scalar EB inequality is said to be \emph{sharp} if its deviation term $D_n$ satisfies the oracular limit \eqref{eqn:sharp}: its first order term, including constants, asymptotically matches the oracle Bernstein inequality which requires knowledge of $\sigma^2$. We see that while \eqref{eqn:eb-informal} is sharp, \eqref{eqn:mp-eb} is not sharp.

Indeed, these two methods are inherently different and as argued convincingly by \citet[Appendix A.8]{howard2021time}: the latter's avoidance of the union bound produces a better concentration.   \cite{waudby2020estimating} were the first ones to prove that their EB inequality is sharp, pointing out that the union bound-based inequalities are not sharp (but only slightly so). We further discuss this issue in \Cref{sec:ub}, showing that one can make the Maurer-Pontil inequality \eqref{eqn:mp-eb} sharp by using a smarter union bound, but it still is empirically looser than \eqref{eqn:eb-informal}. 
Other EB inequalities have been proved in the literature in between the above sets of papers, but they are even looser than the original, so we omit them.

%\begin{equation}\label{eqn:eb}
%    \Pr \left( \text{there exists $n \ge 1$,} \quad  \sum_{i=1}^n \lambda_i (X_i - \Exp X_i)  \ge \log(1/\alpha) + \sum_{i=1}^n 4 ( X_i - \widehat{X}_{i-1} )^2 \psiE(\lambda_i)  \right) \le \alpha,
%\end{equation}
%where each $\lambda_i$ is an arbitrary $(0,1)$-valued ``weighing'' variable measurable w.r.t.\ $\sigma(X_1, \dots, X_{i-1})$, and $\widehat{X}_{i-1}$ an arbitrary plug-in ``prediction'' of $X_i$ measurable w.r.t.\ $\sigma(X_1, \dots, X_{i-1})$, and $\psiE(\lambda):=(-\log(1-\lambda)-\lambda)/4$ for $\lambda\in[0,1)$. 
%Note that \eqref{eqn:eb} is \emph{time-uniform} as it provides a bound on the deviation probability for all times $n \ge 1$ simultaneously. One can take the following steps to see that \eqref{eqn:eb} is indeed an empirical Bernstein inequality: (1) Remove ``there exists $n \ge 1$'' and fix a particular $n$ instead; (2) Take $\lambda_i = \sqrt{\frac{2\log(1/\alpha)}{n()}}$

\subsection{Our Contributions: Matrix Empirical Bernstein Inequalities}\label{sec:mat-eb}

Exponential concentration inequalities for the sum of independent matrices are in general much harder to obtain. 
% The Chernoff-Cram\'er moment generating function (MGF) method that underlies much of the scalar concentration literature faces a direct obstacle when applied to matrices: the matrix MGF of the sum of independent matrices is no longer the product of individual matrix MGF. 
\citet[Theorem 6.1]{tropp2012user} proved a series of Bennett-Bernstein inequalities for the average $\overline\mX_n$ of independent $d\times d$ symmetric matrices $\mX_1,\dots, \mX_n$ with common mean $\Exp \mX_i = \mM$, common eigenvalue upper bound $\lambda_{\max}(\mX_i) \le B$, and $ \sum_{i=1}^n\Exp \mX_i^2 \matle n \mV$. For example, the Bennett-type result implies the following ($\| \cdot \|$ being the spectral norm),
\begin{equation}\label{eqn:mbi}
    \Pr\left( \lambda_{\max} \left( \overline\mX_n -\mM \right) \ge \frac{ B \log(d/\alpha)   }{3n} + \sqrt{  \frac{2 \log(d/\alpha) \|\mV\|  }{n}}  \right) \le \alpha.
\end{equation}
The analogy between \eqref{eqn:bi} and \eqref{eqn:mbi} is straightforward to notice, including matching constants. See \Cref{sec:rmk-bb} for some remarks on these two non-empirical Bernstein results and a proof of \eqref{eqn:mbi}. We shall explore some of the techniques by \cite{tropp2012user} later when developing our results.

%While vector generalizations of empirical Bernstein inequalities exist 
%To the best of our knowledge, no explicit matrix empirical Bernstein inequalities exist in the literature. 
%It is possible, 
%we note, to derive matrix empirical Bernstein inequalities via an empirical variance plug-in to \eqref{eqn:mbi} with a union bound, following the argument by \cite{audibert2009exploration,maurer2009empirical} for scalars. We do it in \Cref{sec:ub-matrix}. As mentioned earlier however, this leads to a slightly sub-optimal result that is not sharp.
The main contribution of the current paper is two empirical Bernstein inequalities for matrices derived in analogy to the two methods in the scalar case mentioned earlier. First, we generalize the union bound and plug-in techniques by \cite{audibert2009exploration,maurer2009empirical} to matrices and obtain:
\begin{proposition}[\Cref{thm:mp-mbi} of this paper, shortened]\label{prop:mpinformal}
 Let $n$ be even and $\mX_1,\dots , \mX_n$ be i.i.d.\ symmetric matrices with eigenvalues in $[0,1]$ and mean $\mM$. Let $\mV_n^*$ be the paired sample variance  $n^{-1}((\mX_1 - \mX_2 )^2+ (\mX_3 - \mX_4 )^2+ \dots + (\mX_{n-1} - \mX_n )^2)$. Then,
\begin{equation} \label{eqn:meb1-informal}
    \Pr\left( \lambda_{\max} \left( \overline\mX_n -\mM \right) \ge  \sqrt{  \frac{2  \|  \mV^*_n  \|  \log \frac{nd}{(n-1)\alpha}  }{n}}  + \mathcal{O}\left(\frac{ \log(nd/\alpha)}{ n
  }\right)\right) \le \alpha.
\end{equation}
\end{proposition}

Second, we provide a faithful generalization of \eqref{eqn:eb-informal} to the matrix case which we informally state as follows.
\begin{proposition}[\Cref{cor:empb-fixn} of this paper, shortened]\label{prop:informal}
    Let $\mX_1,\dots , \mX_n$ be symmetric matrices with eigenvalues in $[0,1]$ and common conditional mean $\mM = \Exp[ \mX_i | \mX_1,\dots, \mX_{i-1} ]$. 
    For an appropriate weighted average $\hat \mM_n$ of $\mX_1,\dots, \mX_n$ and an appropriate sample variance proxy $v_{n,\alpha} =v(\alpha, \mX_1,\dots, \mX_n) > 0$,
    \begin{equation}\label{eqn:meb-informal}
    \Pr\left( \lambda_{\max}(\hat\mM_n -  \mM) \ge \sqrt{\frac{2 \log(d/\alpha) v_{n,\alpha}} {n}}   \right) \le \alpha.
\end{equation}
%If $\{\mX_i\}$ have a constant conditional mean $\Exp[ \mX_i | \mX_{1}, \dots, \mX_{i-1} ] = \mM$,  we have $\Exp \hat \mM_n = \mM$.
    Further, if $\{\mX_i\}$ have a common conditional variance $\mV$, $v_{n,\alpha}$ converges almost surely to 
    $\|\mV \|$. 
\end{proposition}
The detailed description of $\hat \mM_n$ and $v_{n,\alpha}$ will be furnished in \Cref{sec:res}. As in the scalar case, we say a one-sided matrix EB inequality is \emph{sharp} if its deviation term $D_n$ satisfies
%From the statements above, it can also be seen that both \eqref{eqn:meb1-informal} and \eqref{eqn:meb-informal} match \eqref{eqn:mbi} asymptotically without requiring knowing a bound on the largest eigenvalue of the variance, with deviation bounds $D_n$ (the right hand sides of the inequalities) attaining the very same limit
\begin{equation}
    \sqrt{n} D_n \stackrel{a.s.}\to \sqrt{2 \log(d/\alpha) \| \mV \| },
\end{equation}
 as does that of the oracle inequality \eqref{eqn:mbi}. We see that both results above are sharp matrix empirical Bernstein inequalities.
It is also worth remarking that both \eqref{eqn:eb-informal} and our matrix generalization \eqref{eqn:meb-informal} are special fixed-time cases of some \emph{time-uniform} concentration inequalities 
%that are as tight as \eqref{eqn:eb-informal} and \eqref{eqn:meb-informal} at the particular sample size $n$ but provide non-trivial deviation bounds for all other sample sizes within the same $\alpha$ probability,
that control the tails of all $\{\hat \mu_n\}_{n \ge 1}$ or $\{ \hat \mM_n \}_{n \ge 1}$ simultaneously,
enabling sequential analysis. This will become clear as we develop our results.

In applications e.g.\ covariance estimation, our matrix EB inequalities can lead to tremendous improvements over the oracle matrix Bernstein inequality \eqref{eqn:mbi}.
To the best of our knowledge, the only other matrix EB inequality in the literature is the contemporaneous result by \citet[Corollary 3.5]{kroshnin2024bernstein}, which is not sharp.
Besides the work cited above, some other authors have also contributed to the literature of Bernstein or empirical Bernstein inequalities. These include EB inequalities for vectors by \cite{chugg2023time},  for Banach space elements by \cite{martinez2024empirical}; a time-uniform oracle matrix Bernstein inequality by \cite{howard2021time}; a dimension-free version of \eqref{eqn:mbi} by \cite{MINSKER2017111}; and another oracle matrix Bernstein inequality by \cite{mackey2014matrix}. Some of these will be discussed in \Cref{sec:comp}. We also discuss some other closed-form scalar EB inequalities by \cite{tolstikhin2013pac,mhammedi2019pac,mhammedi2021risk,jang2023tighter,orabona-up,shekhar2023near} in \Cref{sec:others}.

%\cite{chugg2023time}, for example, apply the PAC-Bayes technique to the aforementioned self-normalization method from \cite{howard2021time} to obtain an empirical Bernstein inequality for bounded random vectors. \cite{martinez2024empirical} used different techniques to derive a sharp empirical Bernstein inequality in smooth Banach spaces. Neither implies a satisfactory matrix bound.
%In the other direction, \cite{howard2021time} also provide a time-uniform recipe for oracle matrix Bernstein inequality; and \cite{MINSKER2017111} proves a dimension-free alternative to \eqref{eqn:mbi}, replacing $d$ with the smaller ``effective rank'' $\tr(\mV)/\|\mV\|$ but incurring a larger constant. Other matrix Bernstein results in the literature include the one by \cite{mackey2014matrix}. These will be discussed more in \Cref{sec:comp}. Other closed-form EB inequalities for scalars have been proved in the literature in between the above sets of papers, but they are often even looser than the original ones, so we discuss them in \Cref{sec:others}.

%\todo{I hope we can have a bound that depends on effective dimension-dependent as well}

\section{Preliminaries}

\paragraph{Notation} Let $\cS_d$ denote the set of all $d\times d$ real-valued symmetric matrices, which is the only class of matrices considered in this paper. These matrices are denoted by bold upper-case letters $\mA, \mB$, etc.
%; whereas $d$ dimensional (column) vectors by bold lower-case letters $\vu, \vv$, etc. 
For $I \subseteq \mathbb R$, we denote by $\cS_d^I$ the set of all real symmetric matrices whose eigenvalues are all in $I$.  $\cS_d^{ [0,\infty) }$, the set of positive semidefinite and $\cS_d^{ (0,\infty) }$, the set of positive definite matrices are simply denoted by $\cS_d^{+}$ and $\cS_d^{++}$ respectively. 
The Loewner partial order is denoted $\matle$, where $\mA \matle \mB$ means $\mB - \mA$ is positive semidefinite, and $\mA \matls \mB$ means $\mB - \mA$ is positive definite. We use $\lambda_{\max}$ to denote the largest eigenvalue of a matrix in $\cS_d$, and $\| \cdot \|$ its spectral norm, i.e., the largest absolute value of eigenvalues. As is standard in matrix analysis, a scalar-to-scalar function $f:I \to J$ is identified canonically with a matrix-to-matrix function $f: \cS_d^I \to \cS_d^J$, following the definition
\begin{equation}
   f: \mU\trsp \operatorname{diag}[\lambda_1,\dots, \lambda_d] \mU \mapsto \mU\trsp \operatorname{diag}[f(\lambda_1),\dots, f(\lambda_d)] \mU.
\end{equation}
Matrix powers $\mX^k$, logarithm $\log \mX$, and exponential $\exp\mX$ are common examples. It is worth noting that the monotonicity of $f:I \to J$ is usually \emph{not} preserved when lifted to $f: \cS_d^I \to \cS_d^J$ in the $\matle$ order. The matrix logarithm, however, is monotone. On the other hand, for any monotone $f:I\to J$, the function $\tr \circ f : \cS_d^{I} \to \mathbb R$ is always monotone.

%Some basic facts in matrix analysis that we shall use include: (1) the monotonicity of the matrix logarithm, $\mA \matle \mB \implies \log \mA \matle \log \mB$ for all $\mA, \mB \in \cS_{d}^{++}$.

We work on a filtered probability space $(\Omega, \cF, \Pr)$ where $\cF :=  \{\cF_n\}_{n \geq 1}$  is a filtration, and we assume $\cF_0 := \{ \varnothing, \Omega \}$.  We say a process $X := \{ X_n \}$ is adapted if $X_n$ is $\cF_n$-measurable for all integers $n \ge 0$ or sometimes $n \ge 1$; predictable if $X_n$ is $\cF_{n-1}$-measurable for all integers $n\ge 1$.

\paragraph{Nonnegative Supermartingales}
Many of the classical concentration inequalities for both scalars and matrices are derived via Markov's inequality. \cite{howard2020time} pioneered using \emph{Ville's inequality} for \emph{nonngative supermartingales} to construct time-uniform concentration inequalities. An adapted scalar-valued process $\{X_n\}_{n\ge 0}$ is called a nonnegative supermartingale if $X_n \ge 0$ and $\Exp[  X_{n+1} | \cF_{n} ] \le X_n$ for all $n \ge 0$ (all such inequalities are intended in the $\Pr$-almost sure sense). Let us state the following two well-known forms of Ville's inequality, both generalizing Markov's inequality.
\begin{lemma}[Ville's inequality]\label{lem:vi} Let $\{X_n\}$ be a nonnegative supermartingale and $\{ Y_n \}$ be an adapted process such that $Y_n \le X_n$ for all $n$. For any $\alpha \in (0,1]$,
  \begin{equation}
    \textstyle  \Pr\left( \sup_{n \ge 0} Y_n \ge X_0/\alpha  \right) \le \alpha.
  \end{equation} 
  Equivalently, for any stopping time $\tau$,
   \begin{equation}
  \textstyle     \Pr\left(  Y_\tau \ge  X_0/\alpha  \right) \le \alpha.
  \end{equation}  
\end{lemma}
%In \Cref{lem:uvi}, an independent randomization factor $U$ (whose distribution is referred to as ``super-uniform'')  is introduced to further tighten the inequality, a result recently proved by \cite{ramdas2023randomized}.

\paragraph{Matrix CGF Supermartingales}
The Chernoff-Cram\'er MGF method cannot be directly applied to the sum of independent random matrices due to $\exp(\mA + \mB) \neq  (\exp \mA) (\exp \mB)$ in general. \cite{tropp2012user} introduced the method of controlling the \emph{trace} of the matrix CGF via an inequality due to \cite{lieb1973convex}. The Lieb-Tropp method is later furthered by \cite{howard2020time} in turn to construct a nonnegative supermartingale for matrix martingale differences. We slightly generalize it as follows.

\begin{lemma}[Lemma 4 in \cite{howard2020time}, rephrased and generalized]\label{lem:howardlieblemma} Let $\{ \mZ_n \}$ be an $\cS_d$-valued, adapted martingale difference sequence. Let $\{\mC_n\}$ be an $\cS_d$-valued adapted process, $\{\mC_n'\}$ be an $\cS_d$-valued predictable process. If 
\begin{equation}\label{eqn:howard-general}
   \Exp (\exp( \mZ_n -  \mC_n) | \cF_{n-1}) \matle \exp(  \mC_n'  ),
\end{equation}
holds for all $n$, then the process
\begin{equation}\label{eqn:howard-gen-nsm}
  \textstyle  L_n = \tr \exp \left( \sum_{i=1}^n \mZ_i -  \sum_{i=1}^n  ( \mC_i + \mC_i' )  \right)
\end{equation}
is a nonnegative supermartingale. Further,
\begin{equation}\label{eqn:howard-gen-e-proc}
\textstyle   L_n \ge \exp \left( \lambda_{\max}\left(\sum_{i=1}^n \mZ_i \right)-   \lambda_{\max}\left( \sum_{i=1}^n ( \mC_i + \mC_i' ) \right)  \right).
\end{equation}
\end{lemma}

We remark that in the supermartingale \eqref{eqn:howard-gen-nsm}, since the empty sum is the zero matrix, $L_0 = \tr \exp 0 = \tr \mI = d$. This will translate into the $\log(d)$-type dimension dependence in our bounds. The above lemma is proved in \Cref{sec:pfhwl}.

\section{First Matrix EB Inequality: Plug-In}\label{sec:ub-matrix}

The scalar EB inequality \eqref{eqn:mp-eb} by \citet[Theorem 4]{maurer2009empirical} is derived via a union bound between the non-empirical Bennett-Bernstein inequality \eqref{eqn:bi} and a lower tail bound on the Bessel-corrected sample standard deviation. We slightly deviate from their construction by restricting the sample size $n$ to even numbers (discarding an observation if $n$ is odd) and considering the following ``paired'' variance estimator
\begin{equation}\label{eqn:paired}
    \mV^*_n = \frac 1 n (  (\mX_1 -\mX_2)^2 + (\mX_3 -\mX_4)^2 + \dots + (\mX_{n-1} -\mX_n)^2).
\end{equation}
Our first matrix EB inequality follows from applying the non-empirical matrix Bennett-Bernstein inequality \emph{twice}, to the sample average and the paired variance estimator above.

\begin{theorem}[First matrix empirical Bernstein inequality]\label{thm:mp-mbi} Let $n \ge 2$ be even and $\mX_1,\dots, \mX_n$ be $\cS_d^{[0,1]}$-valued independent random matrices with common mean $\mM$ and variance $\mV$. We denote by $\overline{\mX}_n$ their sample average and by $   \mV^*_n$ their paired variance estimator defined in \eqref{eqn:paired}.
    Then, for any $\alpha \in (0,1)$,
    \begin{equation} 
    \Pr\left( \lambda_{\max} \left( \overline\mX_n -\mM \right) \ge   D_n^{\mathsf{meb1}} \right) \le \alpha,
\end{equation}
where
\begin{equation}\label{eqn:meb1}
    D_n^{\mathsf{meb1}} = \frac{ \log\frac{nd}{(n-1)\alpha}   }{3n} + \sqrt{  \frac{2  \| \mV^*_n \| \log\frac{nd}{(n-1)\alpha}  }{n}} + \left( \sqrt{\frac{5}{3}} + 1 \right) \frac{\sqrt{\left(\log\frac{nd}{(n-1)\alpha}\right)\left( \log \frac{2nd}{\alpha}\right) } }{n}.
\end{equation}
Further, if $\mX_1,\dots, \mX_n$ are i.i.d.,
\begin{equation}\label{eqn:asymp-meb1}
    \lim_{n\to\infty} \sqrt{n}   D_n^{\mathsf{meb1}} = \sqrt{2\log(d/\alpha) \|\mV\| }, \quad \text{almost surely}.
\end{equation}
\end{theorem}

\begin{proof}[Proof Sketch]
The inequality follows from the following lower tail bound on $\|  \mV^*_n \|^{1/2}$
\begin{equation}
  \Pr\left(  \|\mV\|^{1/2}  \le \|  \mV^*_n \|^{1/2} + \left(\sqrt{\frac{5}{6}} + \frac{1}{\sqrt{2}} \right) \sqrt{  \frac{\log(2d/\alpha)   }{n}} \right) \ge 1-\alpha,
\end{equation}
and the matrix Bennett-Bernstein inequality \eqref{eqn:mbi} via an $\alpha = \alpha(n-1)/n + \alpha/n$ union bound. 
The full proof can be found in \Cref{sec:pf-1steb}.
\end{proof}

The first order term of the deviation radius \eqref{eqn:meb1} matches the oracle matrix Bernstein inequality \eqref{eqn:mbi}, both being the $\Theta\left(\sqrt{n^{-1} \|\mV \| \log(d/\alpha)}\right)$ variance term. More importantly, the match is precise asymptotically, as is indicated by the limit \eqref{eqn:asymp-meb1} of $\sqrt{n} D_n^{\mathsf{meb1}}$. It is therefore a sharp matrix EB inequality by our standard.
Indeed, this owes much to the imbalanced $\alpha = \alpha(n-1)/n + \alpha/n$ split in the union bound in the proof; if a balanced, or more generally $n$-independent split was employed, the limit would become $\sqrt{2\log(Cd/\alpha) \|\mV \|}$ for some constant $C>1$ instead. A balanced split, however, is exactly what \cite{maurer2009empirical} do in their scalar EB inequality (as well as \cite{kroshnin2024bernstein} concurrently in their matrix EB inequality), leading to the intralogarithmic % I like this neologism!
factor $C=2$ as shown in \eqref{eqn:mp-eb}. 
%This too 
This non-sharpness of the scalar Maurer-Pontil inequality \eqref{eqn:mp-eb}, we remark,
can be avoided as well by switching to the $\alpha = \alpha(n-1)/n + \alpha/n$ imbalanced split instead, which leads to both theoretical sharpness and a significant boost in its large-sample empirical performance. We write it down formally and perform comparative simulations in \Cref{sec:ub}. 

Beyond sharpness,
the second order term of the deviation radius \eqref{eqn:meb1} also has the same rate as the oracle matrix Berntein inequality \eqref{eqn:mbi} up to a logarithmic factor in $n$, both being the boundedness term that decays as $\Tilde{\Theta}(n^{-1})$ as the sample size $n$ grows.
The second order term of \eqref{eqn:meb1} also matches the second order term of the \emph{sharpened} Maurer-Pontil inequality derived in \Cref{sec:ub}, with only a slight inflation of the constant. We remark that \cite{maurer2009empirical} employ a self-bounding technique on the ``classical'' Bessel-corrected scalar sample variance $\hat \sigma_n^2$, and we are not aware such a technique exists for random matrices, leading us to opt for the paired sample variance \eqref{eqn:paired}. In \Cref{sec:classical-sample-variance}, we derive an alternative Maurer-Pontil-style matrix EB inequality using the classical matrix Bessel-corrected sample variance and bound it via the matrix Efron-Stein technique due to \cite{paulin2016efron}. The resulting matrix EB inequality is still sharp due to the similar first order term, but its second order term inflates from $\Tilde{\Theta}(n^{-1})$ to $\Tilde{\Theta}(n^{-3/4})$ and has a slightly worse empirical performance.

%We further remark that a $\mathcal{O}( n^{-1} \| \mV \|^{-1/2} \wedge n^{-3/4} )$ dependence exists in the boundedness term. In terms of \emph{convergence} (i.e., large sample behavior as $n\to\infty$), this is faster than the $\mathcal{O}(n^{-1/2})$ boundedness term with Hoeffding, and the same compared to the $\mathcal{O}(n^{-1})$ rate as in the scalar EB \eqref{eqn:mp-eb}. However, in the small sample regime, since $\|\hat \mV_n \|$ can be arbitrarily small, $\mathcal{O}(n^{-3/4})$ dominates instead and this may be worse than the $\mathcal{O}(n^{-1})$ scalar EB \eqref{eqn:mp-eb} rate.
%This is largely due to the technique we use: our Efron-Stein argument leads to a slower convergence of $\| \widehat{\mV}_n \|^{1/2}$ to $\| {\mV} \|^{1/2}$, compared to the one from the self-bounding technique of \cite{maurer2009empirical}. If there is a matrix self-bounding concentration inequality available that leads to
  %  \begin{equation}
  %  \Pr( \|  \mV - \widehat{\mV}_n \| \ge t  ) \le \mathcal{O}( d) \cdot \exp\left(\frac{-nt^2}{\mathcal{O}(\|\mV \| )} \right),
%\end{equation}
%then we may eliminate the extra factor. We leave this to future work.

\section{Second Matrix EB Inequality: The Supermartingale Method}\label{sec:res}

Our second matrix EB inequality avoids the analysis of the sample variance via the exponential supermartingale technique and opens up for dependent matrices.
Let us, following \cite{howard2020time,howard2021time,waudby2020estimating}, define the function $\psiE:[0,1) \to [0,\infty)$ as $\psiE(\gamma) = -\log(1-\gamma) - \gamma$. The symbol $\psiE$ is from the fact that it is the cumulant generating function (CGF) of a centered standard exponential distribution. The following lemma is a matrix generalization of \citet[Appendix A.8]{howard2021time}, which we prove in \Cref{sec:pfcond}

\begin{lemma}\label{lem:cond-lem}
     Let $\{\mX_n\}$ be an adapted sequence of $\cS_d$-valued random matrices with conditional means $\Exp (\mX_n | \cF_{n-1}) = \mM_n$. Further, suppose there is a predictable and integrable sequence of $\cS_d$-valued random matrices $\{ \widehat \mX_n \}$ such that $\lambda_{\min}(\mX_n - \widehat \mX_n ) \ge -1$. Let
\begin{equation}
     \mE_n =  \exp(\gamma_n(\mX_n - \widehat{\mX}_n)  - \psiE(\gamma_n) (\mX_n - \widehat{\mX}_n)^2) ,\quad \mF_n = \exp(\gamma_n (\mM_n - \widehat{\mX}_n)  ),
\end{equation}
where $\{ \gamma_n \}$ are predictable $(0,1)$-valued scalars. Then,
\begin{equation}
    \Exp (\mE_n|\cF_{n-1}) \matle \mF_{n}.
\end{equation}
\end{lemma}

We are now ready to state in full our matrix empirical Bernstein inequality based on the self-normalization technique. The following theorem is stated as a combination of three tools: a nonnegative supermartingale, a time-uniform concentration inequality, and an equivalent %randomized 
concentration inequality at a stopping time.

\begin{theorem}[Time-uniform and stopped matrix empirical Bernstein inequalities]\label{thm:empb}
     Let $\{\mX_n\}$ be an adapted sequence of $\cS_d$-valued random matrices with conditional means $\Exp (\mX_n|\cF_{n-1}) = \mM_n$. Let $\{ \widehat\mX_n \}$ be a sequence of predictable and integrable $\cS_d$-valued random matrices such that $\lambda_{\min}(\mX_n - \widehat{\mX}_n) \ge -1$ almost surely.
     Then, for any 
     %stopping time $\tau$, 
     predictable $(0,1)$-valued sequence $\{\gamma_n\}$, 
     \begin{equation}\label{eqn:eb-nsm}
    L_n^{\mathsf{meb2}} = \tr \exp\left( \sum_{i=1}^n \gamma_i(\mX_i - \mM_i) - \sum_{i=1}^n \psiE(\gamma_i) (\mX_i - \widehat{\mX}_i)^2 \right)
\end{equation}
is a supermartingale. Denote by $\overline{\mX}_{n}^\gamma$ the weighted average
$
    \frac{\gamma_1 \mX_1 + \dots + \gamma_n \mX_n }{\gamma_1 + \dots + \gamma_n}$
w.r.t.\ the positive weight sequence $\{ \gamma_n \}$. Then, for any $\alpha\in(0,1)$,
\begin{equation}\label{eqn:empb-tu}
         \Pr\left(  \text{there exists }n\ge 1, \;    \lambda_{\max} \left( \overline{\mX}_n^\gamma - \overline{\mM}_n^\gamma \right) \ge \frac{  \log(d/\alpha) +\lambda_{\max}\left( \sum_{i=1}^n \psiE(\gamma_i) (\mX_i - \widehat{\mX}_i)^2 \right)  }{\gamma_1 + \dots + \gamma_n} \right) \le \alpha;
     \end{equation}
and for any stopping time $\tau$,  $\alpha\in(0,1)$, 
%and super-uniform scalar random variable  $U$ independent of $\cF_\infty$,
     \begin{equation}\label{eqn:empb}
         \Pr\left( \lambda_{\max} \left( \overline{\mX}_\tau^\gamma - \overline{\mM}_\tau^\gamma \right) \ge \frac{  \log(d/\alpha) +\lambda_{\max}\left( \sum_{i=1}^\tau \psiE(\gamma_i) (\mX_i - \widehat{\mX}_i)^2 \right)  }{\gamma_1 + \dots + \gamma_\tau} \right) \le \alpha.
     \end{equation}
\end{theorem}

\begin{proof}
Due to \Cref{lem:cond-lem}, we can apply \Cref{lem:howardlieblemma} with $\mZ_n = \gamma_n(\mX_n - \mM_n)$, $\mC_n = \gamma_n(\widehat{\mX}_n - \mM_n ) + \psiE(\gamma_n) (\mX_n - \widehat{\mX}_n)^2$, and $\mC_n' = \gamma_n (\mM_n - \widehat{\mX}_n)$ to see that
\begin{equation}
  \textstyle  L_n^{\mathsf{meb2}} = \tr \exp\left( \sum_{i=1}^n \gamma_i(\mX_i - \mM_i) - \sum_{i=1}^n \psiE(\gamma_i) (\mX_i - \widehat{\mX}_i)^2 \right)
\end{equation}
is a supermartingale, which upper bounds
\begin{equation}\label{eqn:epr}
 \textstyle   \exp \left\{ \lambda_{\max}\left(\sum_{i=1}^n \gamma_i(\mX_i - \mM_i) \right)-   \lambda_{\max}\left( \sum_{i=1}^n \psiE(\gamma_i) (\mX_i - \widehat{\mX}_i)^2 \right)  \right\}.
\end{equation}
Applying \Cref{lem:vi} to \eqref{eqn:epr}, the desired result follows from rearranging.
\end{proof}

Before we remark on the uncompromised \Cref{thm:empb}, let us write down its fixed-time, fine-tuned special case of \eqref{eqn:empb} with $\tau = n$ and conditionally homoscedastic observations. This
 shall justify the ``empirical Bernstein" name it bears.

\begin{corollary}[Second matrix empirical Bernstein inequality]\label{cor:empb-fixn}
Suppose $\alpha \in (0,1)$.
    Let $\mX_1,\dots,\mX_n$ be adapted $\cS_d^{[0,1]}$-valued random matrices with constant conditional mean $\mM = \Exp (\mX_i|\cF_{i-1})$ and constant conditional variance $\mV = \Exp(  (\mX_i - \mM)^2|\cF_{i-1} )$.  Let $\overline{\mX}_i = \frac{1}{i}(\mX_1+\dots+\mX_i)$ and $\overline{\mX}_0 = 0$.
     Define the following variance proxies
    \begin{gather}
       \overline{\mV}_0 = \frac
       1 4 \mI, \quad  \overline{\mV}_k = \frac{1}{k}\sum_{i=1}^k (\mX_i - \overline\mX_k)^2, \quad  \overline{v}_k = \|\overline{\mV}_{k}\| \vee \frac{5\log(d/\alpha)}{n},
    \end{gather}
    and set $\gamma_i = \sqrt{\frac{2\log(d/\alpha)}{n \overline{v}_{i-1}}}$ for $i=1,\dots, n$. Then,
    \begin{equation}\label{eqn:empb-fix-n}
         \Pr\left( \lambda_{\max} \left( \overline{\mX}_n^{\gamma} - \mM \right) \ge D_n^{\mathsf{meb2}}  \right) \le \alpha, \, \text{where} \ D_n^{\mathsf{meb2}} = \frac{  \log(d/\alpha) +\lambda_{\max}\left( \sum_{i=1}^n \psiE(\gamma_i) (\mX_i - \overline{\mX}_{i-1})^2 \right)  }{\gamma_1 + \dots + \gamma_n}.
     \end{equation}
    Further, %letting
    %\begin{equation}
     %   \widetilde{s}_n =\lambda_{\max}  \left( \frac{1}{n}\sum_{i=1}^n \frac{ (\mX_i - \overline\mX_{i-1})^2}{\overline{v}_{i-1}}  \right),
    %\end{equation}
    %then 
    %nonasymptotically,
     %   \begin{equation}
      %      D_n^{\mathsf{meb2}} \le   \sqrt{\frac{\log(d/\alpha)}{2n}} \frac{1+ 2\widetilde{s}_n }{ \frac{1}{n}\sum_{i=1}^n \overline{v}_{i-1}^{-1/2}} \quad %\stackrel{\text{a.s.}}\sim 
       %     \approx \sqrt{\frac{9\log(d/\alpha) \|\mV \| }{2n}};
       % \end{equation}
        %and 
        asymptotically,
        \begin{equation}\label{eqn:ebasymp}
          \lim_{n \to \infty} \sqrt{n}   D_n^{\mathsf{meb2}} = \sqrt{2\log(d/\alpha) \|\mV\|} \quad \text{almost surely}.
        \end{equation}
       % where $D_n^{\mathsf{B}} = \sqrt{\frac{2\log(d/\alpha) \lambda_{\max}(\mI + 2\mV) }{3n}}$ is the  deviation bound of the optimal matrix Bernstein inequality \eqref{eqn:mat-bern-fixn} requiring known variance $\mV = \Var (\mX_1)$.
\end{corollary}
We prove \Cref{cor:empb-fixn} in \Cref{sec:pffix}. 
The asymptotic behavior \eqref{eqn:ebasymp} of deviation bound $ D_n^{\mathsf{meb2}}$ is satisfying as it \emph{adapts fully} to, without knowing, the true variance $\mV$. In particular, if the assumption on the known spectral bound is $\mX_1,\dots, \mX_n \in \cS_d^{[a,b]}$ as opposed to the $\cS_d^{[0,1]}$ stated in \Cref{cor:empb-fixn}, one can apply the result to $\frac{\mX_1 - a}{b-a}, \dots ,\frac{\mX_n - a}{b-a}$ to obtain the same
\begin{equation}
    {\Theta}\left( \sqrt{\frac{\log(d/\alpha) \|\mV\|}{n}}  \right)
\end{equation}
asymptotic deviation which is free of $b-a$.

The three kinds of result stated in \Cref{thm:empb} are for potentially different purposes. The supermartingale \eqref{eqn:eb-nsm} is best as a sequential test for the null 
\begin{equation}
  H_0: \  \Exp(\mX_n | \cF_{n-1}) = \mM_{  \textsf{null} }  \quad \text{for all }n
\end{equation}
by setting each $\mM_i$ to $\mM_{  \textsf{null} }$. The time-uniform concentration inequality \eqref{eqn:empb-tu} can be used to construct a ``confidence sequence'' on the common conditional mean $\mM = \Exp(\mX_n | \cF_{n-1})$; that is, a sequence of confidence balls $B_n = \{ \mM' \in \cS_d : \|  \overline{\mX}_n^{\gamma} -  \mM'  \| \le \rho_n  \}$ such that $\Pr( \mM \in \cap_n B_n  ) \ge 1-\alpha$, leading to the stopped concentration inequality \eqref{eqn:empb} which is a valid confidence ball at a fixed stopping time $B_{\tau}$. We also remark that it is possible to sharpen the confidence ball $B_{\tau}$ at a fixed stopping time by an \emph{a priori} randomization, due to a recent result by \citet[Theorem 4.1]{ramdas2023randomized} called ``uniformly randomized Ville's inequality''. That is, letting  $U\sim \text{Unif}_{(0,1)}$ independent from the filtration $\cF$, one may replace the $\log(d/\alpha)$ term in \eqref{eqn:empb} with the strictly smaller $\log(Ud/\alpha)$.

%This new result, which we prove in \Cref{sec:pfeb}, recovers exactly the sequential scalar empirical Bernstein inequality \citet[Theorem 2]{waudby2020estimating}  as the $d=1$ special case, which in turn generalizes \citet[Theorem 4]{howard2021time}.
%We reiterate our remark following \Cref{ex:emp-bern-ms} 
The $\cF_{i-1}$-measurable term $\widehat{\mX}_i$ in Theorem~\ref{thm:empb} is best understood as a ``plug-in prediction'' of the next observation $\mX_i$. Indeed, whereas the inequality holds under all choices of $\widehat{\mX}_i$, the smaller the ``prediction error'' $(\widehat{\mX}_i  - \mX_i)^2$, the tighter the bound. Thus one may set $\widehat{\mX}_i$ to be the sample average from $\mX_1$ to $\mX_{i-1}$, which is exactly what is done in \Cref{cor:empb-fixn}. 

On the other hand, if the sample size $n$ is \emph{not} fixed in advance and
an infinite sequence of i.i.d.\ (or homoscedastic more generally) observations $\mX_1, \mX_2,\dots$, to construct a tight time-uniform concentration bound or powerful sequential test, we recommend setting the weight sequence $\{ \gamma_n \}$ as follows: for each sample size $n$, temporarily assume that a sample size of $n$ is fixed in advance, compute the optimal choice of weight on $\mX_n$ under this fixed sample size, and set $\gamma_n$ to this optimal choice of weight. This will lead to a vanishing sequence of $\gamma_n = \sqrt{\frac{2\log(d/\alpha)}{n \overline{v}_{n-1}}}$. Under this weight sequence, we see the choice of a \emph{weighted average} 
\begin{equation}
    \widehat{\mX}_n = \overline{\mX}_{n-1}^{\psiE(\gamma)} = \frac{  \sum_{i=1}^{n-1}  \psiE(\gamma_i) \mX_i }{ \sum_{i=1}^{n-1}  \psiE(\gamma_i) }
\end{equation}
is more reasonable as the weighted sum of squares $\sum_{i=1}^n \psiE(\gamma_i) (\mX_i - x)^2$ is minimized by the weighted average
\begin{equation}
     \hat x = \frac{  \sum_{i=1}^{n}  \psiE(\gamma_i) \mX_i }{ \sum_{i=1}^{n-1}  \psiE(\gamma_i) }.
\end{equation}
Of course, as long as $\widehat \mX_n$ is any convex combination of $\mX_1,\dots,\mX_{n-1}$, the condition $\lambda_{\min}(\mX_n - \widehat \mX_n ) \ge -1$ is met when $\{\mX_n\}$ all take values in $\cS_d^{[0,1]}$.

Finally, as a reprise of the shortened version \Cref{prop:informal} stated in the opening, the ``approprioate variance proxy'' $v_{n,\alpha} = v(\alpha, \mX_1,\dots, \mX_n)$ is simply
\begin{equation}
    v_{n,\alpha}  =  \left(\frac{  \log(d/\alpha) +\lambda_{\max}\left( \sum_{i=1}^n \psiE(\gamma_i) (\mX_i - \overline{\mX}_{i-1})^2 \right)  }{\gamma_1 + \dots + \gamma_n}\right)^2 \frac{n}{2\log(d/\alpha)}
\end{equation}
which converges almost surely to $\|\mV \|$ under conditional homoscedasticity due to \eqref{eqn:ebasymp}.
%Finally, noting that $\psiE(\gamma_n) \approx \gamma_n^2/2$ for small $\gamma_n$, one can see that \eqref{eqn:empb} and \eqref{eqn:usmhi} match exactly in their scales: when, for example, $\gamma_n \asymp n^{-1/2}$, both deviation bounds scale as $\sqrt{\log n / n}$ with i.i.d.\ or homoscedastic observations.

\section{Comparisons} \label{sec:comp}

\subsection{Self-Normalized EB Inequalities for Scalars and Vectors}

Our \Cref{thm:empb} and \Cref{cor:empb-fixn} owe much to the techniques developed by \citet[Theorem 2 and Remark 1]{waudby2020estimating} in the scalar case (who in turn build on the earlier result by \citet[Theorem 4]{howard2021time} via the ``predictable mixing" sequence $\{ \gamma_n \}$). In particular, when $d=1$, our statements match (including constants) exactly the scalar empirical Bernstein inequality counterparts by \cite{waudby2020estimating}: Our supermartingale \eqref{eqn:eb-nsm} coincides with Equation (13) in \cite{waudby2020estimating}; our time-uniform concentration bound \eqref{eqn:empb-tu} becomes identical to Theorem 2 in \cite{waudby2020estimating}; and our fixed-time asymptotics \eqref{eqn:ebasymp} recovers Equation (17) in \cite{waudby2020estimating}. We also note that \cite{waudby2020estimating} assume i.i.d.\-ness to obtain the fixed-time asymptotics, which, according to our result, can be relaxed to martingale dependence.

As can be expected, applying a vector bound to matrices (by flattening) will lead to a very suboptimal result. The self-normalized empirical Bernstein inequality for vectors due to \citet[Corollary 5]{chugg2023time} implies the following for matrices whose Frobenius norm is bounded by $1/2$, for all $\alpha \le 0.1$,
\begin{equation}
    \Pr\left(\| \hat\mM_n -  \mM \|_{\text{F}} \ge 3.25 \sqrt{\frac{  \log(1/\alpha) \tilde \sigma_n^2 } {n}}   \right) \le \alpha.
\end{equation}
Here, $\tilde \sigma_n^2$ converges almost surely to the vectorized variance $\Exp \| \mX_1 - \Exp \mM  \|^2_{\text{F}}$ with i.i.d.\ matrices. Since everything (assumption and result) is in the Frobenius norm, however, translating the result into the spectral norm will incur a dimensional dependence polynomial in $d$.

Finally, we note that the self-normalized empirical Bernstein inequality for Banach spaces due to \cite{martinez2024empirical} is not applicable as $\cS_d$ equipped with the spectral norm is not a 2-smooth Banach space.

\subsection{Non-Empirical Matrix Bernstein and Hoeffding Inequalities}

As we state in the opening \eqref{eqn:mbi} and elaborate further in \Cref{sec:rmk-bb}, \citet[Theorem 1.4]{tropp2012user} proves the following matrix Bennett-Bernstein inequality under the assumptions $\max_{1\le i \le n} \lambda_{\max}(\mX_i) \le 1$ and $ \sum_{i=1}^n\Exp \mX_i^2 \matle  n\mV$:
\begin{equation}\label{eqn:Dtb}
    \Pr\left( \lambda_{\max} \left( \overline\mX_n - \Exp \overline\mX_n \right) \ge D_n^{\mathsf{tb}}  \right) \le \alpha, \quad D_n^{\mathsf{tb}} =  \frac{ B \log(d/\alpha)   }{3n} + \sqrt{  \frac{2 \log(d/\alpha) \|\mV\|  }{n}} .
 %   \frac{  \log(d/\alpha) + \sqrt{ \log^2(d/\alpha) + 18n\log(d/\alpha) \lambda_{\max}(\mV)}  }{3n}.
\end{equation}
We can see that with i.i.d.\ matrices with variance $\mV$, $\sqrt{n}   D_n^{\mathsf{tb}}$ converges to $\sqrt{2\log(d/\alpha) \|\mV\| }$ which is the same limit that both $\sqrt{n}   D_n^{\mathsf{meb1}}$ and  $\sqrt{n}   D_n^{\mathsf{meb2}}$ converge to, stated as \eqref{eqn:asymp-meb1} and \eqref{eqn:ebasymp}. Therefore, our empirical Bernstein inequalities provide a confidence region fully adaptive to the unknown variance $\mV$ and match in asymptotics this oracle Bernstein result which requires $\mV$ to be known. Both are thus \emph{sharp} EB inequalities.
Assumption-wise, it is important to note that it is fair to compare our $\mX_i \in \cS_d^{[0,1]}$ assumption to their $\lambda_{\max}(\mX_i) \le 1$ assumption; no constant is glossed over in making this comparison when and two-sided bound is sought. To see this, the bound by \citet[Theorem 1.4]{tropp2012user} can be applied to $\mX_1 - \mM, \dots, \mX_n - \mM$, and it takes $\mX_1 \in \cS_d^{[0,1]}$ to ensure both $\lambda_{\max}(\mX_1 - \mM) \le 1$ and $\lambda_{\max}(-\mX_1 +\mM) \le 1$ hold.

\citet[Corollary 5.2]{mackey2014matrix} also obtain a matrix Bernstein inequality. However, as they acknowledge in the paper, their bound is strictly looser than the bound by \citet[Theorem 1.4]{tropp2012user}. The bound by \citet[Theorem 3.1]{MINSKER2017111} under the same assumption reads
\begin{equation} \small
    \Pr\left( \lambda_{\max} \left( \overline\mX_n - \Exp \overline\mX_n \right) \ge D_n^{\mathsf{mb}}  \right) \le \alpha, \quad D_n^{\mathsf{mb}} = \frac{ B \log(d'/\alpha) + \sqrt{ 
B^2 \log^2(d'/\alpha) + 18n\log(d'/\alpha) \|\mV\|}  }{3n},
\end{equation}
where $d' = 14 \tr(\mV) / \|\mV \|$, which decides the dimension-free virtue of their result. %This can be tighter than \citet[Theorem 1.4]{tropp2012user} only if the largest eigenvalue is at least 14 times greater than the average eigenvalue of $\mV$. 
Matrix Bernstein inequalities that are either anytime-valid or empirical remain an open problem.

Finally, we quote the tightest known Hoeffding-type inequalities for matrices in the literature.   \citet[Corollary 4.2]{mackey2014matrix} shows that if independent $\mX_1, \dots, \mX_n$ satisfy $(\mX_i - \Exp \mX_i)^2  \matle \mB$  almost surely, then
\begin{equation}\label{eqn:tropphoef}
  \Pr\left( \lambda_{\max} (\overline{\mX}_n - \Exp \overline{\mX}_n ) \ge  \sqrt{\frac{2 \|\mB \| \log(d/\alpha)}{n} } \right) \le \alpha.
\end{equation}
A time-uniform extension can be achieved by applying Lemma 3(h) in \cite{howard2020time}, but its fixed-time corollary remains identical as \eqref{eqn:tropphoef}. 
The squared boundedness assumption $(\mX_i - \Exp \mX_i)^2  \matle \mB$ implies $\mX_i - \Exp \mX_i \in \cS_d^{ [ - \|\mB\|^{1/2}, \|\mB\|^{1/2} ]   }$, so it is a stronger assumption than the boundedness assumption we make in \Cref{cor:empb-fixn}. Further, 
since  $(\mX_i - \Exp \mX_i)^2  \matle \mB$ implies $\Var (\mX_i) \matle \mB$ and in practice this gap can be arbitrarily large, we see that our empirical Bernstein inequality is asymptotically tighter and the \emph{worst} that can happen is a degradation to this, already tightest, matrix Hoeffding bound, when $\|\Var (\mX_i) \| \approx \|\mB\|$.

The advantage of matrix EB inequalities becomes even clearer when we consider the recurring application example in the matrix concentration literature: covariance matrix estimation.
\begin{example}[Adaptive covariance estimation]
    Let $X_1,\dots, X_n$ be i.i.d.\ random vectors in $\mathbb R^d$ with $\|X_1 \| \le 1$ almost surely, mean $\Exp X_1 = 0$ and covariance matrix $\Exp (X_1 X_1\trsp) = \mSg$. Since
    \begin{equation}
        \lambda_{\max}(X_1 X_1\trsp) = \tr(X_1 X_1\trsp) = \tr (X_1\trsp X_1 ) \le 1, 
    \end{equation}
    we can invoke either of the two matrix EB inequalities (\cref{thm:mp-mbi} or \cref{cor:empb-fixn}) with $\mX_i = X_i X_i\trsp$ and $\mM = \mSg$ to construct confidence sets for $\mSg$ that are adaptive to the unknown 4th moment $\Exp ( \|X_1\|^2 X_1  X_1\trsp)$.
\end{example}

In comparison, covariance estimation bounds derived via \eqref{eqn:mbi} or its variants always have the unknown $\| \mSg \|$ term in the radius of the concentration, from bounding the 4th moment $\Exp ( \|X_1\|^2 X_1  X_1\trsp) \matle \Exp (X_1  X_1\trsp) = \mSg$ \citep{tropp2015introduction,howard2021time}. To turn them into nonasymptotic confidence sets, a further boundedness argument $\|\mSg\| \le 1$ is required, essentially reducing the Bernstein bound to a Hoeffding bound which can be arbitrarily loose. The sharpness of our EB-based methods immediately benefits the downstream applications of covariance estimation in e.g.\ machine learning and signal processing.

\subsection{Simulation}

We compare the terms, $ D_n^{\mathsf{meb1}}$ of the first matrix empirical Bernstein inequality as in \eqref{eqn:meb1}, and $D_n^{\mathsf{meb2}}$ of the second matrix empirical Bernstein inequality as in \eqref{eqn:empb-fix-n}, divided by that of the oracle matrix Bennett-Bernstein inequality $D_n^{\mathsf{tb}}$ as in
\eqref{eqn:Dtb}. We set $\alpha$ to .05, thus comparing the tightness of one-sided 95\%-confidence sets (or equivalently, the spectral diameters of two-sided 90\%-confidence sets). The i.i.d.\ random matrices are generated from 3 fixed orthonormal projections with $d=3$, each with an independent $\operatorname{Unif}_{[0,1]}$ eigenvalue. The comparison is displayed in \Cref{tab:mat-comp}. We see that while for both matrix EB inequalities, the deviations relative to the oracle Bernstein are proved to converge to 1 as $n$ increases, our second matrix EB inequality achieves such ``oracular convergence'' at a much faster rate.  A comparison in the scalar case can be found in \Cref{sec:ub} which conveys a similar message.
\begin{table}[!h]
    \centering
    \begin{tabular}{|c|c|c|} 
    \hline
        Sample Size ($n$) & $D_n^{\mathsf{meb1}}/D_n^{\mathsf{tb}}$ 
        & $D_n^{\mathsf{meb2}}/D_n^{\mathsf{tb}}$  \\ \hline
       100  & 2.612  & 1.397 \\
       1,000 & 1.589  & 1.105 \\
       10,000 & 1.214  & 1.022 \\
       100,000 & 1.072  & 1.007 \\
       1,000,000 & 1.024  & 1.002 \\ \hline
    \end{tabular}
    \caption{Relative lengths of 95\% one-sided confidence sets by two sharp matrix empirical Bernstein inequalities compared to the (oracle) matrix Bennett-Bernstein inequality \eqref{eqn:mbi}.}
    \label{tab:mat-comp}
\end{table}

\section{Summary}

We provide two new matrix concentration inequalities in this paper. The first one is based on the union bound method, and characterizes, in terms of the paired sample variance, the concentration of the sample mean of independent symmetric matrices with bounded largest eigenvalues, common mean, and common variance. The second one is a self-normalized, time-uniform concentration inequality for the weighted sum of martingale difference symmetric matrices with bounded largest eigenvalues, which when weighted properly, becomes an empirical Bernstein inequality that echoes many of the previous self-normalized-type empirical Bernstein inequalities for scalar, vectors, and Banach space elements.

These two matrix EB inequalities have different advantages: the first one is conceptually simpler, requires only the sample mean and a sample variance,
and is computationally easier (both needing $\mathcal{O}(n)$ steps but the first one having smaller constants);
the second EB is empirically tighter across all sample sizes, allows martingale dependence, and
has a time-uniform version.
According to our simulation, the second matrix EB inequality's confidence set is only 10.5\% larger compared to the oracle Bernstein inequality under a sample size of 1,000; and 2.2\% under a sample size of 10,000.
On the other hand, our two matrix EB inequalities both have a closed-form expression, and they
 match in asymptotics the best non-empirical matrix Bernstein inequality in the literature, as they only depend (in the large sample limit) on the true variance of the matrices which is not required to be known in our bounds, but required in non-empirical bounds.
We expect future work to address 
%the relatively minor problem of further eliminating the $\mathcal{O}(\|\mV\|^{-1/2} \wedge n^{1/4} )$ lower-order extra dependence in \eqref{eqn:meb1}, and 
the challenging problem of unifying our methods with those of the dimension-free matrix Bernstein inequality by \cite{MINSKER2017111}.

%% file: eb/ack.tex
We thank Diego Martinez-Taboada and Arun Kumar Kuchibhotla for helpful discussions. We thank an anonymous referee for pointing out an improvement for \Cref{sec:ub-matrix}.

%% file: eb/appdx.tex
\section{Additional Proofs}

\subsection{Technical Lemmas}

The following lemma converts bounds on $|a-b|$ to $\sqrt{a} - \sqrt{b}$.

\begin{lemma}\label{lem:sqrt}
    Let $a, b \ge 0$ and $D = |a-b|$. Then,
    \begin{equation}
        \sqrt{a} \le \sqrt{b} + \left( \sqrt{D} \wedge \frac{D}{2\sqrt{b}} \wedge \frac{D}{\sqrt{a}} \right).
    \end{equation}
\end{lemma}
\begin{proof} Suppose $a > b$ since  the bound is trivial otherwise.
    First, by the subadditivity of the square root, $\sqrt{a} = \sqrt{b+D} \le \sqrt{b} + \sqrt{D}$. Second,
    $D = (\sqrt{a} - \sqrt{b})(\sqrt{a} + \sqrt{b}) \ge (\sqrt{a} - \sqrt{b})2\sqrt{b}$ so $\sqrt{a} \le \sqrt{b} + \frac{D}{2\sqrt{b}}$.  Third,  $D = (\sqrt{a} - \sqrt{b})(\sqrt{a} + \sqrt{b}) \ge (\sqrt{a} - \sqrt{b})\sqrt{a}$ so $
        \sqrt{a} \le \sqrt{b} +  \frac{D}{\sqrt{a}}$.
Taking a minimum completes the proof.
\end{proof}

The following lemma bounds the difference of squares of two matrices.
\begin{lemma}\label{lem:mat-dfsq-bd}
  Let $\mA,\mB$ be two symmetric matrices. $\| \mA^2 - \mB^2 \| \le 2 \| \mB \| \| \mA - \mB  \| + \| \mA - \mB  \|^2$.
\end{lemma}
\begin{proof} $
         \| \mA^2 - \mB^2 \| = \| (\mB + (\mA - \mB))^2 - \mB^2   \|
        =  \|    \mB (\mA - \mB) + (\mA-\mB)\mB +  (\mA-\mB)^2  \|
        \le  2 \| \mB \| \| \mA - \mB  \| + \| \mA - \mB  \|^2$.
\end{proof}

The following lemma characterizes the smoothness of  $\psiE(x) = -\log(1-x)-x$ at 0.

\begin{lemma}\label{lem:psiE-smoothness}
    When $0 \le x \le \sqrt{2/5}$, $\psiE(x) \le x^2$.
\end{lemma}
\begin{proof} Let $g(x) = \psiE(x) -  x^2$. The claim follows from $g''(x) = (1-x)^{-2} - 2 \ge 0$ for $x \in [0, \sqrt{2/5}]$, and $g(0) = 0$, $g(\sqrt{2/5}) < 0$.
\end{proof}

The following \emph{transfer rule} \citep[Equation 2.2]{tropp2012user} is commonly used in deriving matrix bounds.

\begin{lemma}\label{lem:transfer}
Suppose $I \subseteq \mathbb R$ and $f, g:I \to \mathbb R$ satisfies $f(x) \le g(x)$, then, $f(\mX) \matle g(\mX)$ for any $\mX \in \cS_d^I$.
\end{lemma}

It is well-known that if $X$ and $Y$ are scalar random variables such that $c \le X \le Y$ almost surely for some constant $c$ and that $\Exp |Y| < \infty$, it follows that $\Exp |X| < \infty$ as well, and $\Exp X \le \Exp Y$. This type of ``implied integrability" appears frequently in scalar concentration bounds.
Let us prove its symmetric matrix extension for the sake of self-containedness.

\begin{lemma}[Dominated integrability]\label{lem:domint}
    Let $\mX$ and $\mY$ be $\cS_d^{[c,\infty)}$-valued random matrices for some $c\in \mathbb R$ such that $\mX \matle \mY$ almost surely. Further, suppose $\Exp \mY$ exists. Then, so does $\Exp \mX$ and $\Exp \mX \matle \Exp \mY$. 
\end{lemma}
\begin{proof} Let us prove that each element $X_{ij}$ of the random matrix $\mX$ is integrable.
    Note that for any deterministic $\vv \in \mathbb R^d$, $\vv \trsp \mX \vv \le \vv \trsp \mY \vv$ almost surely. First, taking $\vv = (0,\dots, 0, 1 , 0,\dots 0)\trsp$, we have
    \begin{equation}
        c \le X_{jj} \le Y_{jj} \quad \text{almost surely},
    \end{equation}
    concluding that the diagonal element $X_{jj}$ must be integrable (since $Y_{jj}$ is). Next, taking $\vv = (0,\dots, 0, 1 , 0,\dots ,0, 1 , 0, \dots, 0)\trsp$, we have
    \begin{equation}
       2 c \le 2X_{ij} + X_{ii} + X_{jj} \le 2Y_{ij} + Y_{ii} + Y_{jj} \quad \text{almost surely},
    \end{equation}
    concluding that $2X_{ij} + X_{ii} + X_{jj}$ must be integrable (since $2Y_{ij} + Y_{ii} + Y_{jj}$ is). Therefore, the off-diagonal element $X_{ij}$ is integrable since $X_{ii}$ and $X_{jj}$ are.

    Now that we have established the existence of $\Exp \mX$, it is clear that $\Exp \mX \matle \Exp \mY$ since for any $\vv \in \mathbb R^d$, $
        \vv\trsp (\Exp \mX) \vv = \Exp  (\vv\trsp \mX \vv) \le \Exp  (\vv\trsp \mY \vv) = \vv\trsp (\Exp \mY) \vv$.
\end{proof}

\subsection{Strong Consistency of the Matrix Sample Variance}\label{sec:consistency}

We show in this section that the matrix sample mean and variance are strongly consistent under martingale dependence, which prepares us for the upcoming proof of \Cref{cor:empb-fixn}. First, we provide the following matrix martingale strong law of the large numbers which, we remark, holds for non-square matrices as well.

\begin{lemma}[Matrix martingale SLLN]\label{lem:mat-slln} Let $\mS_n = \sum_{i=1}^n \mZ_i$ be a matrix martingale and $\{U_n \}$ an increasing positive predictable process on $\cF$. For any $p\in[1,2]$, on the set
\begin{equation}
    \left\{ \lim_{n \to \infty} U_n = \infty,  \; \sum_{n=1}^\infty U_n^{-p} \Exp[ \operatorname{norm}^p(\mZ_n) | \cF_{n-1} ] < \infty \right\},
\end{equation}
the process $U_n^{-1} \mS_n$ converges to 0 almost surely. Here, $\operatorname{norm}(\cdot)$ is any matrix norm.
\end{lemma}

\begin{proof}
   Since matrix norms are mutually equivalent, it suffices to prove the case where $\operatorname{norm}(\mA)$ is the max norm $ \max_{i,j} |\mA_{ij}|$. By the scalar martingale SLLN \citep[Theorem 2.18]{hall2014martingale}, we see that all entries of $\mS_n$ converge to 0 almost surely.
\end{proof}

This immediately implies the following convergence result of the matrix sample mean and variance, where we take  $\operatorname{norm}(\cdot)$ to be the usual spectral norm $\| \cdot \|$.

\begin{corollary}[Strong consistency of the matrix sample mean]\label{cor:consistency-mean} Let $\{ \mX_n \}$ be adapted to $\cF$ with constant conditional mean $\Exp[\mX_n|\cF_{n-1}] = \mM$ and $\delta \in (0,1]$. If $\sum_{n=1}^\infty n^{-1-\delta} \Exp[ \|\mX_n - \mM \|^{1+\delta} | \cF_{n-1} ] < \infty $ almost surely, then the sample mean $\overline{\mX}_n$ converges to $\mM$ almost surely.
\end{corollary}
\begin{proof}
    It follows directly from \Cref{lem:mat-slln}.
\end{proof}
The assumption $\sum_{n=1}^\infty n^{-1-\delta} \Exp[ \|\mX_n - \mM \|^{1+\delta} | \cF_{n-1} ] < \infty $ is satisfied when, for example, all matrices are uniformly bounded or have a common $(1+\delta)$\textsuperscript{th} conditional moment upper bound.

\begin{corollary}[Strong consistency of the matrix sample variance]\label{cor:consistency-var}
    Let $\mX_n$ be adapted to $\cF$ with constant conditional mean $\Exp[\mX_n|\cF_{n-1}] = \mM$ and conditional variance $\Exp[(\mX_n - \mM)^2|\cF_{n-1}] = \mV$. Further, assume that
    \begin{equation}
        \sum_{n=1}^\infty n^{-1-\delta} \Exp[ \|\mX_n \|^{2+2\delta} | \cF_{n-1}  ] < \infty
    \end{equation}
    almost surely for some $\delta \in (0,1]$.
    Then, the sample variance
    \begin{equation}
        \overline \mV_n = \frac 1 n \sum_{i=1}^n (\mX_i - \overline{\mX}_n)^2
    \end{equation}
    converges to $\mV$ almost surely.
\end{corollary}
\begin{proof} Define $\mQ =  \Exp[ \mX_n^2 | \cF_{n-1} ] = \mM^2 + \mV$. First,
$\overline{\mX}_n \to \mM$ almost surely due to the constant variance with \Cref{cor:consistency-mean}. Using the inequality $\| \mA + \mB \|^{1+\delta} \le ( \|\mA\| + \| \mB \|  )^{1+\delta} \le 2^{\delta}( \|\mA\|^{1+\delta} + \|\mB\|^{1+\delta} ) $ (due to triangle and Jensen inequalities), we have
\begin{equation}
   \sum_{n=1}^\infty n^{-1-\delta} \Exp[ \| \mX_n^2 - \mQ \|^{1+\delta} | \cF_{n-1} ] \le  \sum_{n=1}^\infty n^{-1-\delta}  2^{\delta}\Exp[  \| \mX_n\|^{2+2\delta} | \cF_{n-1}   ] + \sum_{n=1}^\infty n^{-1-\delta}  2^{\delta} \mQ^{1+\delta} < \infty.
\end{equation}
So by \Cref{cor:consistency-mean}.
\begin{equation}
   \hat {\mQ}_n  := \frac{1}{n}\sum_{i=1}^n \mX_i^2 \to  \mQ .
   \end{equation}
almost surely. 
Therefore,
\begin{equation}
     \overline \mV_n =    \hat {\mQ}_n  - (\overline\mX_n)^2 \to \mV
\end{equation}
almost surely as well, due to continuity.
\end{proof}
Again, the assumption $  \sum_{n=1}^\infty n^{-1-\delta} \Exp[ \|\mX_n \|^{2+2\delta} | \cF_{n-1}  ] < \infty $ is satisfied when, for example, all matrices are uniformly bounded or have a common $(2+2\delta)$\textsuperscript{th} conditional moment upper bound.

\subsection{Proof of \Cref{lem:howardlieblemma}}\label{sec:pfhwl}
\begin{proof}
    Due to the monotonicity of $\log$, the condition \eqref{eqn:howard-general} implies
    \begin{equation}\label{eqn:howard-general-log}
    \log \Exp (\exp( \mZ_n -  \mC_n) | \cF_{n-1}) \matle \mC_n'.
\end{equation}
    Now recall Lieb's concavity theorem \citep{lieb1973convex}: for any $\mH \in \cS_d$, the map $\mX \mapsto \tr \exp(\mH + \log \mX)$ ($\cS_d^{++} \to (0,\infty)$) is concave. Therefore,
    \begin{align}
       \Exp (L_n|\cF_{n-1}) 
        =  \ &  \Exp \left( \tr \exp \left( \sum_{i=1}^{n-1} \mZ_i -  \sum_{i=1}^{n-1}( \mC_i + \mC_i' )  - \mC_n' + \log   \e^{ \mZ_n -  \mC_n  }  \right) \middle| \cF_{n-1} \right)
       \\
       & \text{(Jensen's inequality)} \nonumber
       \\
       \le \ &   \tr \exp \left( \sum_{i=1}^{n-1}\mZ_i -  \sum_{i=1}^{n-1} ( \mC_i + \mC_i' ) -\mC_n' + \log \Exp  (\e^{ \mZ_n - \mC_n   } |\cF_{n-1}) \right) 
       \\
       & \text{(by \eqref{eqn:howard-general-log} and monotonicity of trace)} \nonumber
       \\
       \le \ &  \tr \exp \left( \sum_{i=1}^{n-1} \mZ_i -  \sum_{i=1}^{n-1}  ( \mC_i + \mC_i' )  - \mC_n' + \mC_n'  \right) = L_{n-1},
    \end{align}
    concluding the proof that $\{ L_n \}$ is a supermartingale.
    Finally, observe that
    \begin{align}
        L_n & =  \tr \exp \left( \sum_{i=1}^n \mZ_i -  \sum_{i=1}^n  ( \mC_i + \mC_i' )  \right)
        \\
        & \ge   \tr \exp \left( \sum_{i=1}^n \mZ_i -  \lambda_{\max} \left(\sum_{i=1}^n  ( \mC_i + \mC_i' ) \right) \mI \right)
        \\
        & \ge  \lambda_{\max}   \exp \left( \sum_{i=1}^n \mZ_i -  \lambda_{\max} \left(\sum_{i=1}^n  ( \mC_i + \mC_i' ) \right) \mI \right) 
        \\
        & = \exp \lambda_{\max} \left( \sum_{i=1}^n \mZ_i -  \lambda_{\max} \left(\sum_{i=1}^n  ( \mC_i + \mC_i' ) \right) \mI \right) 
        \\
        & =  \exp \left( \lambda_{\max}\left(\sum_{i=1}^n \mZ_i \right)-   \lambda_{\max}\left( \sum_{i=1}^n  ( \mC_i + \mC_i' ) \right)  \right),
    \end{align}
    concluding the proof.
\end{proof}

\subsection{Proof of \Cref{thm:mp-mbi}}\label{sec:pf-1steb}

\begin{proof}
    Note that
$2 \mV^*_n$ is the sample average of $n/2$ independent random matrices, each with eigenvalues in $[0, 1]$, of common mean $2\mV$ and second moment upper bound
\begin{equation}
   \| \Exp  (\mX_1 -\mX_2)^4 \| \le  \| \Exp  (\mX_1 -\mX_2)^2 \| = 2 \| \mV \|.
\end{equation}
Thus applying the Matrix Bennett-Bernstein inequality \eqref{eqn:mbi} on $2\mV_n^*$, we see that
\begin{equation}
    \Pr\left( \| 2\mV^*_n - 2\mV \| \ge \frac{  2\log(2d/\alpha)   }{3n} + \sqrt{  \frac{8 \log(2d/\alpha) \|\mV\|  }{n}}  \right) \le \alpha.
\end{equation}
Therefore, with probability at least $1-\alpha$,
\begin{equation}
    | \  \|  \mV^*_n \| -   \|\mV\|  \ | \le \| \mV^*_n - \mV \| \le \frac{  \log(2d/\alpha)   }{3n} + \sqrt{  \frac{2 \log(2d/\alpha) \|\mV\|  }{n}}.
\end{equation}
Denote by $g$ the constant $\sqrt{\frac{5}{6}} + \frac{1}{\sqrt{2}}$, which satisfies $\frac{1}{3g} + \sqrt{2} = g$. The above implies that if $\| \mV \| \ge \frac{g^2 \log(2d/\alpha)}{n}$,
\begin{equation}
    | \  \|  \mV^*_n \| -   \|\mV\|  \ | \le  \sqrt{\frac{  \log(2d/\alpha) \| \mV \|  }{9g^2 n}} + \sqrt{  \frac{2 \log(2d/\alpha) \|\mV\|  }{n}}
     = \sqrt{   \frac{g^2 \log(2d/\alpha) \|\mV\|  }{n}},
\end{equation}
which in turn implies that, via \Cref{lem:sqrt},
\begin{equation}
    \|\mV\|^{1/2}  \le \|  \mV^*_n \|^{1/2} + \sqrt{  \frac{g^2 \log(2d/\alpha)   }{n}}.
\end{equation}
Since the above also holds if $\| \mV \| < \frac{g^2 \log(2d/\alpha)}{n}$, we see that the inequality above holds with probability at least $1-\alpha$ regardless of the true value of $\| \mV \|$.
Integrating the inequality above into the matrix Bennett-Bernstein inequality \eqref{eqn:mbi} via an $\alpha = \alpha(n-1)/n + \alpha/n$ union bound, we arrive at
\begin{equation}\small
   \lambda_{\max} \left( \overline\mX_n -\mM \right) \le \frac{ \log\frac{nd}{(n-1)\alpha}   }{3n} + \sqrt{  \frac{2 \log\frac{nd}{(n-1)\alpha}   }{n}}
    \left(\|  \mV^*_n \|^{1/2} + \left( \sqrt{\frac{5}{6}} + \frac{1}{\sqrt{2}}  \right)\sqrt{  \frac{\log(2nd/\alpha)   }{n}} \right)
\end{equation}
with probability at least $1-\alpha$. This
concludes the proof. The asymptotics \eqref{eqn:asymp-meb1} follows simply from the strong law of large numbers and the continuity of the spectral norm.
\end{proof}

\subsection{Proof of \Cref{lem:cond-lem}}\label{sec:pfcond}

\begin{proof}
Recall that $\psiE(\gamma) = -\log(1-\gamma) - \gamma$.  An inequality by \cite{fan2015exponential} quoted by \citet[Appendix A.8]{howard2021time} states that, for all $0 \le \gamma < 1$ and $\xi \ge -1$,
\begin{equation}
    \exp(\gamma \xi - \psiE(\gamma) \xi^2) \le 1 + \gamma \xi.
\end{equation}
Since $\mX_n - \widehat{\mX}_n \in \cS_d^{[-1,\infty)}$, we can apply the transfer rule (\Cref{lem:transfer}), replacing the scalar $\xi$ above by the matrix $\mX_n - \widehat{\mX}_n$, and plugging in $\gamma = \gamma_n \in (0,1)$,
\begin{equation}
    \exp(\gamma_n(\mX_n - \widehat{\mX}_n)  - \psiE(\gamma_n) (\mX_n - \widehat{\mX}_n)^2) \matle 1 + \gamma_n (\mX_n - \widehat{\mX}_n).
\end{equation}
\Cref{lem:domint} then guarantees the integrability of the left hand side, and that
\begin{align}
  & \Exp\left( \exp(\gamma_n(\mX_n - \widehat{\mX}_n)  - \psiE(\gamma_n) (\mX_n - \widehat{\mX}_n)^2) \middle|\cF_{n-1} \right) \matle  \Exp\left( 1 + \gamma_n (\mX_n - \widehat{\mX}_n) \middle|\cF_{n-1} \right) 
   \\
   = & 1 + \gamma_n (\mM_n - \widehat{\mX}_n)  \matle \exp(\gamma_n (\mM_n - \widehat{\mX}_n)  ),
\end{align}
where in the final step we use the transfer rule again with $1+x\le \exp(x)$ for all $x\in\mathbb R$. This concludes the proof.
\end{proof}

\subsection{Proof of \Cref{cor:empb-fixn}}\label{sec:pffix}

\begin{proof}
    First, it is straightforward that $\lambda_{\max}({\mX_i - \overline{\mX}_{i-1}}) \ge -1$ for every $i = 1,\dots, n$ since both $\mX_i$ and $\overline{\mX}_{i-1}$ take values in $\cS_d^{[0,1]}$, so \Cref{thm:empb} is applicable. Let us prove the two claims about the deviation bound $D_n^{\mathsf{meb2}}$ under $\gamma_i = \sqrt{\frac{2\log(d/\alpha)}{n \overline{v}_{i-1}}}$. Recall that
    \begin{gather}
      \overline{\mV}_0 = \frac
       1 4 \mI, \quad  \overline{\mV}_k = \frac{1}{k}\sum_{i=1}^k (\mX_i - \overline\mX_k)^2, \quad  \overline{v}_k = \|\overline{\mV}_{k}\| \vee \frac{5\log(d/\alpha)}{n},
       \\   \widetilde{s}_n =\lambda_{\max}  \left( \frac{1}{n}\sum_{i=1}^n \frac{ (\mX_i - \overline\mX_{i-1})^2}{\overline{v}_{i-1}}  \right).
    \end{gather}
    Let us compute the almost sure limit $\lim_{n \to \infty}   \sqrt{n} D_n^{\mathsf{meb2}}$. First, the following limits hold almost surely due to \Cref{cor:consistency-mean,cor:consistency-var}:
    %(cf.\ \citet[Lemmas 4-6]{waudby2020estimating}):
    \begin{gather}
    \lim_{k\to \infty} \overline{\mX}_k =\mM , \quad
        \lim_{k\to \infty} \overline{\mV}_k =\mV ,  \quad \lim_{k \to \infty} \overline{v}_k = \|\mV \|.
    \end{gather}
Let us compute the limit
\begin{equation}
     \lim_{n \to \infty}   \frac{1}{n}\sum_{i=1}^n 
 \frac{(\mX_i - \overline{\mX}_{i-1})^2 }{ \overline{v}_{i-1}}.
\end{equation}
To do this, we observe that
\begin{equation}
\lim_{n \to \infty}   \frac{1}{n}\sum_{i=1}^n 
 \frac{(\mX_i - \mM)^2 }{ \| \mV \| } = \frac{\mV}{\|\mV\|}
\end{equation}
almost surely due to \Cref{cor:consistency-mean}. Bounding the difference
\begin{align}
    & \left\|  \frac{1}{n}\sum_{i=1}^n 
 \frac{(\mX_i - \overline{\mX}_{i-1})^2 }{ \overline{v}_{i-1}} -  \frac{1}{n}\sum_{i=1}^n 
 \frac{(\mX_i - \mM)^2 }{ \| \mV \| } \right\|
 \\
 \le & \frac{1}{n}\sum_{i=1}^n \left(  \frac{\| (\mX_i - \overline{\mX}_{i-1} )^2 - ( \mX_i - \mM)^2 \| }{\overline{v}_{i-1}} +  \|\mX_i - \mM \|^2|\overline{v}_{i-1}^{-1}  - \| \mV \|^{-1}    |   \right)
 \\
 & \text{(\Cref{lem:mat-dfsq-bd})}
 \\
 \le &  \frac{1}{n}\sum_{i=1}^n \left( \frac{ 2 \| \mX_i - \mM \| \| \overline{\mX}_{i-1} - \mM  \| + \| \overline{\mX}_{i-1} - \mM  \|^2}{\overline{v}_{i-1}} +  \|\mX_i - \mM \|^2 |\overline{v}_{i-1}^{-1}  - \| \mV \|^{-1}    |    \right)
 \\
 & (\| \overline{\mX}_{i-1} - \mM  \| \le 2, \| \mX_i - \mM  \| \le 2)
 \\
 \le & \frac{1}{n} \sum_{i=1}^n \frac{ 6  \| \overline{\mX}_{i-1} - \mM  \| }{\overline{v}_{i-1}} +  \frac{1}{n} \sum_{i=1}^n 4 |\overline{v}_{i-1}^{-1}  - \| \mV \|^{-1}| \to 0.
\end{align}
Therefore, we see that
\begin{equation}
       \frac{1}{n}\sum_{i=1}^n 
 \frac{(\mX_i - \overline{\mX}_{i-1})^2 }{ \overline{v}_{i-1}} \to \frac{\mV}{\| \mV \|}, \quad \widetilde{s}_n =\lambda_{\max}  \left( \frac{1}{n}\sum_{i=1}^n \frac{ (\mX_i - \overline\mX_{i-1})^2}{\overline{v}_{i-1}}  \right) \to 1
\end{equation}
almost surely.

We can now use the expansion $\psiE(x) = \sum_{k=2}^\infty \frac{x^k}{k}$ to obtain,
\begin{align}
   &  \limsup_{n \to \infty} \sqrt{n} D_n^{\mathsf{meb2}}
   \\
   = & \limsup_{n \to \infty}   \frac{  \log(d/\alpha) +\lambda_{\max}\left( \sum_{i=1}^n \psiE\left(\sqrt{\frac{2\log(d/\alpha)}{n \overline{v}_{i-1}}}\right) (\mX_i - \overline{\mX}_{i-1})^2 \right)  }{\frac{1}{n}\sum_{i=1}^n\sqrt{\frac{2\log(d/\alpha)}{\overline{v}_{i-1}}}}
   \\
   \le & \limsup_{n \to \infty}   \frac{  \log(d/\alpha) +\lambda_{\max}\left( \sum_{i=1}^n 
   \frac{1}{2}\left(\sqrt{\frac{2\log(d/\alpha)}{n \overline{v}_{i-1}}}\right)^2 (\mX_i - \overline{\mX}_{i-1})^2 \right)  }{\frac{1}{n}\sum_{i=1}^n\sqrt{\frac{2\log(d/\alpha)}{\overline{v}_{i-1}}}}  
   \\
    + & \sum_{k=3}^\infty \limsup_{n \to \infty} \underbrace{\frac{ \lambda_{\max}\left( \sum_{i=1}^n 
   \frac{1}{k}\left(\sqrt{\frac{2\log(d/\alpha)}{n \overline{v}_{i-1}}}\right)^k (\mX_i - \overline{\mX}_{i-1})^2 \right)  }{\frac{1}{n}\sum_{i=1}^n\sqrt{\frac{2\log(d/\alpha)}{\overline{v}_{i-1}}}}}_{  T_{k, n} } .
\end{align}
Let us prove that $\lim_{n \to \infty} T_{k, n} = 0$ for all $k \ge 3$. To see that,
\begin{align}
    T_{k,n} \le &  \frac{  \sum_{i=1}^n 
   \frac{1}{k}\left(\sqrt{\frac{2\log(d/\alpha)}{n \overline{v}_{i-1}}}\right)^k \|\mX_i - \overline{\mX}_{i-1}\|^2  }{\frac{1}{n}\sum_{i=1}^n\sqrt{\frac{2\log(d/\alpha)}{\overline{v}_{i-1}}}} 
   \le  \frac{  \sum_{i=1}^n 
   \frac{1}{k}\left(\sqrt{\frac{2\log(d/\alpha)}{n \overline{v}_{i-1}}}\right)^k   }{\frac{1}{n}\sum_{i=1}^n\sqrt{\frac{2\log(d/\alpha)}{\overline{v}_{i-1}}}} 
   \\
   = &  \frac{ 
   k^{-1} \left(2\log(d/\alpha)\right)^{\frac{k-1}{2}} n^{-\frac{k-2}{2}} \boxed{ n^{-1} \sum_{i=1}^n 
 \overline{v}_{i-1}^{-k/2}    }
   }{
 \boxed{  n^{-1} \sum_{i=1}^n \overline{v}_{i-1}^{-1/2} }
   }.
\end{align}
Since the boxed terms converge to non-zero quantities, we see that $T_{k,n}$ converges to 0 due to the $n^{- \frac{k-2}{2}}$ term. Therefore,
\begin{align}
     &  \limsup_{n \to \infty} \sqrt{n} D_n^{\mathsf{meb2}}
   \\ 
    \le & \limsup_{n \to \infty}   \frac{  \log(d/\alpha) +\lambda_{\max}\left( \sum_{i=1}^n 
   \frac{1}{2}\left(\sqrt{\frac{2\log(d/\alpha)}{n \overline{v}_{i-1}}}\right)^2 (\mX_i - \overline{\mX}_{i-1})^2 \right)  }{\frac{1}{n}\sum_{i=1}^n\sqrt{\frac{2\log(d/\alpha)}{\overline{v}_{i-1}}}}  
\\
 = & \limsup_{n \to \infty} \sqrt{\frac{\log(d/\alpha)}{2}}  \frac{   \left(1  +\lambda_{\max}\left( \frac{1}{n}\sum_{i=1}^n 
 \frac{(\mX_i - \overline{\mX}_{i-1})^2 }{ \overline{v}_{i-1}} \right) \right) }{\frac{1}{n}\sum_{i=1}^n \overline{v}_{i-1}^{-1/2} } 
 \\
 = & \sqrt{2\log(d/\alpha) \|\mV \|}. 
\end{align}
Similarly, one can show that $ \liminf_{n \to \infty} \sqrt{n}   D_n^{\mathsf{meb2}} \ge \sqrt{2\log(d/\alpha) \|\mV \|}$, concluding the proof. We remark that the proof above strengthens that of \citet[Lemmas 4-8]{waudby2020estimating}.
\end{proof}

\section{Additional Concentration Inequalities}

\subsection{Remarks on the Scalar \eqref{eqn:bi} and Matrix \eqref{eqn:mbi} Bennett-Bernstein Inequalities}\label{sec:rmk-bb}

Non-empirical Bernstein inequalities are typically stated in terms of the upper bound of the tail probability $\Pr( S_n - \Exp S_n \ge t )$. These are derived via Bennett-type inequalities via controlling the function
\begin{equation}
    h(u) = (1+u)\log(1+u) - u \stackrel{(*)}{\ge} \frac{u^2}{2(1+u/3)}.
\end{equation}
We, for statistical purposes however, are interested in deviation bounds under a fixed error probability $\alpha$. The Bennett-to-Bernstein conversion (*) is looser than the following inequality. 

\begin{lemma}\label{lem:audibert-hinv} For all $x\ge 0$,
    $h^{-1}(x) \le \sqrt{2x} + x/3$.
\end{lemma}
A proof of this polynomial upper bound on $h^{-1}$ can be found from Equation (45) onwards in \cite{audibert2009exploration}. \citet[Theorem 6.1]{tropp2012user} first states a matrix Bennett bound in terms of the $h$ function, then uses (*) to obtain a closed-formed matrix Bernstein bound, both controlling the tail probability $\Pr( \lambda_{\max}( \mS_n - \Exp \mS_n) \ge t )$. Let us use \Cref{lem:audibert-hinv} to recover a fixed-error $\alpha$ bound whose tightness is between the matrix Bennett and the matrix Bernstein, which we already recorded in the paper as \eqref{eqn:mbi}.

\begin{proposition}[Matrix Bennett-Bernstein inequality \eqref{eqn:mbi}]
    Let $\overline\mX_n$ be the sample average of independent $d\times d$ symmetric matrices $\mX_1,\dots, \mX_n$ with common mean $\Exp \mX_i = \mM$, common eigenvalue upper bound $\lambda_{\max}(\mX_i) \le B$, and $ \sum_{i=1}^n\Exp \mX_i^2 \matle n \mV$. For any $\alpha > 0$
\begin{equation}
    \Pr\left( \lambda_{\max} \left( \overline\mX_n -\mM \right) \ge \frac{ B \log(d/\alpha)   }{3n} + \sqrt{  \frac{2 \log(d/\alpha) \|\mV\|  }{n}}  \right) \le \alpha.
\end{equation}
\end{proposition}

\begin{proof}
    Due to \citet[Equation (i) in Proof of Theorem 6.1]{tropp2012user}, 
    \begin{equation}
        \Pr\left[  \lambda_{\max} \left( \overline\mX_n -\mM \right) \ge t \right] \le d \cdot \exp \left( - \frac{n \lambda_{\max}(\mV) }{B^2} \cdot h\left(\frac{Bt}{ \lambda_{\max}(\mV)} \right) \right).
    \end{equation}
    Setting the right hand side as $\alpha$, we obtain via \Cref{lem:audibert-hinv}
    \begin{equation}
        t = \frac{\lambda_{\max}(\mV)}{B}  h^{-1}\left( \frac{\log(d/\alpha) B^2}{n \lambda_{\max}(\mV) } \right) \le \frac{\lambda_{\max}(\mV)}{B} \left( \sqrt{\frac{2\log(d/\alpha) B^2}{n \lambda_{\max}(\mV) }} + \frac{\log(d/\alpha) B^2}{3n \lambda_{\max}(\mV) }  \right),
    \end{equation}
    which readily leads to the bound \eqref{eqn:mbi}
    \begin{equation}
    \Pr\left( \lambda_{\max} \left( \overline\mX_n -\mM \right) \ge \frac{ B \log(d/\alpha)   }{3n} + \sqrt{  \frac{2 \log(d/\alpha) \lambda_{\max} (\mV)  }{n}}  \right) \le \alpha. 
\end{equation}
We also remark that the scalar case \eqref{eqn:bi} is when $d=1$.
\end{proof}

\subsection{Sharp Maurer-Pontil Inequality}\label{sec:ub}

\citet[Theorem 4]{maurer2009empirical} derived a scalar empirical Bernstein inequality which we quote as \eqref{eqn:mp-eb}, by a union bound between a scalar Bennett-Bernstein inequality and a tail bound on the sample variance. However, their balanced union bound split $\alpha = \alpha/2 + \alpha/2$ leads to the looser $\log(2/\alpha)$ term. This causes the confidence interval to be 10.9675\% longer when $\alpha= 0.05$ in the large sample limit.
We slightly modify their proof below to obtain a sharp EB inequality for scalars.

\begin{proposition}\label{prop:shmp}
   Let $X_1,\dots, X_n$ be $[0,1]$-bounded independent random scalars with common mean $\mu$ and variance $\sigma^2$. We denote by $\overline{X}_n$ their sample average and $\hat \sigma_n^2$ the Bessel-corrected sample variance. Then, for any $\alpha\in(0,1)$, $
     \Pr\left(  \overline X_n  - \mu  \ge \rho_n   \right) \le \alpha$,
where
\begin{equation}
    \rho_n = \frac{ \log\frac{n}{(n-1)\alpha}  }{3n} + \sqrt{\frac{2 \hat \sigma^2_n \log\frac{n}{(n-1)\alpha} }{n}}  + 2 \sqrt{\frac{\left( \log\frac{n}{(n-1)\alpha}  \right) \left( \log \frac n\alpha  \right)}{n(n-1)}}.
\end{equation}
Further, with i.i.d.\ $X_1,\dots, X_n$,
\begin{equation}
    \lim_{n \to \infty} \sqrt{n} \rho_n =  \sqrt{2 \sigma^2 \log(1/\alpha)}, \quad \text{almost surely}.
\end{equation}
\end{proposition}

\begin{proof}
 By Bennett-Bernstein inequality \eqref{eqn:bi},
\begin{equation}
    \Pr\left(  \overline X_n  - \mu  \ge \frac{ \log(1/\alpha)  }{3n} + \sqrt{\frac{2\sigma^2 \log(1/\alpha)}{n}}    \right) \le \alpha.
\end{equation}
The deviation of $\hat \sigma_n^2$ from $\sigma^2$ is controlled by a self-bounding concentration inequality \citep[Theorem 7]{maurer2009empirical},
\begin{equation}
   \Pr\left(  \sigma - \hat\sigma_n \ge \sqrt{\frac{2 \log(1/\alpha)}{n-1}}  \right) \le \alpha.
\end{equation}
The desired bound thus follows from an $\alpha = \alpha (n-1)/n + \alpha/n$ union bound.
\end{proof}

\begin{table}[!h]
    \centering
    \begin{tabular}{|c|c|c|c|} 
    \hline
        Sample Size ($n$) & 
Original MP
        & Sharp MP & Self-normalized  \\ \hline
       100  & 2.120 & 2.256 & 1.527 \\
       1,000 & 1.441 & 1.476 & 1.081 \\
       10,000 & 1.214 & 1.169 & 1.017 \\
       100,000 & 1.144 & 1.059 & 1.005 \\
       1,000,000 & 1.121 & 1.021 & 1.002 \\ \hline
    \end{tabular}
    \caption{Lengths of 95\% one-sided confidence intervals of three empirical Bernstein inequalities divided by that of the (oracle) Bennett-Bernstein inequality \eqref{eqn:bi}. ``Original MP'' stands for the result of \citet[Theorem 4]{maurer2009empirical}; ``Sharp MP'' our sharpened result \Cref{prop:shmp}; and ``Self-normalized'' the result by \citet[Theorem 2, Remark 1]{waudby2020estimating}.}
    \label{tab:scalar-comp}
\end{table}

As can be seen from the simulation with $\operatorname{Unif}_{[0,1]}$-distributed random variables displayed in \Cref{tab:scalar-comp}, our sharpened EB inequality leads to significant improvement for large samples. Still, the EB inequality based on the self-normalized technique by \cite{howard2021time,waudby2020estimating} has a much smaller gap compared to the oracle Bernstein inequality.

\subsection{First Matrix EB Inequality with the Classical Sample Variance}\label{sec:classical-sample-variance}

Recall in \Cref{sec:ub-matrix} we proved the first matrix EB inequality by applying the matrix Bennett-Bernstein inequality \eqref{eqn:mbi} twice, once on the sample average and once on the paired sample variance $\mV^*_n$. The latter step differs from the scalar result by \cite{maurer2009empirical} which uses a self-bounding technique on the classical Bessel-corrected sample variance.
%and an Efron-Stein bound on the matrix sample variance \eqref{eqn:es-var}. The latter, however, introduces a $n^{-3/4}$ second-order term compared to the $n^{-1}$ second-order term in the scalar case which we discussed in \Cref{sec:ub} above.
%While we stress again that the first matrix 
%EB inequality, \Cref{thm:mp-mbi}, is still \emph{sharp} despite the $n^{-3/4}$ term, we explore in this section a slight variation of the result such that the $n^{-3/4}$ term becomes the $n^{-1}$ that matches the scalar case.
Intuitively however, the classical matrix Bessel-corrected sample variance for $\mX_1,\dots,\mX_n$
\begin{equation}
    \widehat{\mV}_n = \frac{1}{n(n-1)} \sum_{ 1 \le i < j \le n}(\mX_i - \mX_j)^2,
\end{equation}
should have a better convergence rate to the population variance $\mV$ than $\mV^*_n$ as a full U-statistic. This fact is captured by the self-bounding concentration inequality \cite{maurer2009empirical} employ in the scalar case. For matrices, however, the best analysis of $ \widehat{\mV}_n$ we are aware of comes from the Efron-Stein technique due to \cite{paulin2016efron}, which leads to the following variant of \Cref{thm:mp-mbi}.

%We have the following Maurer-Pontil-style matrix EB inequality.
\begin{theorem}[First matrix empirical Bernstein inequality, classical sample variance]\label{thm:mp-mbic} Let $\mX_1,\dots, \mX_n$ be $\cS_d^{[0,1]}$-valued independent random matrices with common mean $\mM$ and variance $\mV$. We denote by $\overline{\mX}_n$ their sample average and $\widehat{\mV}_n$ the Bessel-corrected sample variance.
    Then, for any $\alpha \in (0,1)$,
    \begin{equation} 
    \Pr\left( \lambda_{\max} \left( \overline\mX_n -\mM \right) \ge   D_n^{\mathsf{meb1c}} \right) \le \alpha,
\end{equation}
where
\begin{equation}\label{eqn:meb1c}\small
    D_n^{\mathsf{meb1c}} =  \sqrt{  \frac{2 \log \frac{nd}{(n-1)\alpha}    }{n} } \left( \|\widehat\mV_n\|^{1/2}  +  \sqrt{\frac{\log(2nd/\alpha)}{ 2 n \|\widehat{\mV}_n \| }} \wedge \left( \frac{2\log(2nd/\alpha)}{n} \right)^{1/4}  \right)   + \frac{ \log  \frac{nd}{(n-1)\alpha} }{3n} .
\end{equation}
Further, if $\mX_1,\dots, \mX_n$ are i.i.d.,
\begin{equation}\label{eqn:asymp-meb1c}
    \lim_{n\to\infty} \sqrt{n}   D_n^{\mathsf{meb1c}} = \sqrt{2\log(d/\alpha) \|\mV\| }, \quad \text{almost surely}.
\end{equation}
\end{theorem}

\begin{proof}
 We view the classical sample variance $\widehat{\mV}_n\in \cS_d^{[0,1]}$ as a bounded matrix function of independent variables $\mX_1,\dots,\mX_n$. Let $\widehat{\mV}_n^{j}$ be the sample variance by replacing $\mX_j$ with an i.i.d.\ copy $\mX_j'$. The Efron-Stein variance proxy of $\widehat{\mV}_n$ satisfies
\begin{equation}
    \frac{1}{2}\sum_{j=1}^n \Exp[  (\widehat{\mV}_n - \widehat{\mV}_n^{j})^2 | \mX_1, \dots, \mX_n ] \matle \frac{1}{2n} \mI, %\matle \frac{1}{n} \mI + \frac{1}{2n}(\widehat{\mV}_n - \mV),
\end{equation}
which can be noted from the fact that each $\widehat{\mV}_n - \widehat{\mV}_n^{j} \in \cS_d^{[-1/n, 1/n]}$.
We now invoke the Efron-Stein tail bound, Corollary 5.1 from \cite{paulin2016efron} to see that for any $t > 0$,
\begin{equation}
  \Pr(  | \, \|  \mV \| - \|\widehat{\mV}_n \| \, | \ge t  ) \le  \Pr( \|  \mV - \widehat{\mV}_n \| \ge t  ) \le 2d \exp\left( \frac{-n t^2}{2} \right).
\end{equation}
Setting the right hand side to $\alpha$, we obtain, with probability at least $1-\alpha$,
\begin{equation}\label{eqn:es-var}
   | \, \|  \mV \| - \|\widehat{\mV}_n \| \, | < \sqrt{\frac{2\log(2d/\alpha)}{n}},
\end{equation}
which, due to \Cref{lem:sqrt}, implies that
\begin{equation}
    \|  \mV \|^{1/2} <  \|\widehat{\mV}_n \|^{1/2} + \sqrt{\frac{\log(2d/\alpha)}{ 2n \|\widehat{\mV}_n \| }} \wedge \left( \frac{2\log(2d/\alpha)}{n} \right)^{1/4}.
\end{equation}
Integrating the inequality above into the matrix Bennett-Bernstein inequality \eqref{eqn:mbi} via an $\alpha = \alpha(n-1)/n + \alpha/n$ union bound concludes the proof. The asymptotics \eqref{eqn:asymp-meb1c} follows simply from the strong consistency of the sample variance and the continuity of the spectral norm.
\end{proof}

We extend the previous simulation shown in \Cref{tab:mat-comp} with this additional matrix EB inequality. As can be seen from \Cref{tab:mat-comp-ext}, using the sample variance leads to a slightly larger confidence set despite the asymptotics \eqref{eqn:asymp-meb1c} still being sharp. We leave it to future work whether an improved analysis of the classical matrix sample variance better than \eqref{eqn:es-var} can lead to a Maurer-Pontil-style matrix EB inequality that beats \Cref{thm:mp-mbi}.
\begin{table}[!h]
    \centering
    \begin{tabular}{|c|c|c|c|} 
    \hline
        Sample Size ($n$) & $D_n^{\mathsf{meb1}}/D_n^{\mathsf{tb}}$ & $D_n^{\mathsf{meb1c}}/D_n^{\mathsf{tb}}$
        & $D_n^{\mathsf{meb2}}/D_n^{\mathsf{tb}}$  \\ \hline
       100  & 2.612  & 3.057 & 1.397 \\
       1,000 & 1.589 & 1.874 & 1.105 \\
       10,000 & 1.214 & 1.313 & 1.022 \\
       100,000 & 1.072 & 1.109 & 1.007 \\
       1,000,000 & 1.024 & 1.037 & 1.002 \\ \hline
    \end{tabular}
    \caption{Relative lengths of 95\% one-sided confidence sets by three sharp matrix empirical Bernstein inequalities compared to the (oracle) matrix Bennett-Bernstein inequality \eqref{eqn:mbi}. This table includes an additional column for $D_n^{\mathsf{meb1c}}$ on top of \Cref{tab:mat-comp}.}
    \label{tab:mat-comp-ext}
\end{table}

\section{Other Scalar EB Inequalities}\label{sec:others}

A wealth of scalar empirical Bernstein inequalities exists in the learning theory literature, especially under the PAC-Bayes \citep{alquier2024user} setup. These are concentration inequalities on the posterior expected deviation, valid simultaneously over all posterior distributions. We provide a brief bibliographical remark in this section.

A PAC-Bayes EB inequality was proved by 
\cite{tolstikhin2013pac} using a technique comparable to \cite{audibert2009exploration,maurer2009empirical}: they used a union bound to combine an oracle Bernstein-style PAC-Bayes bound and a sample variance PAC-Bayes bound. When taking a degenerate parameter space (i.e.\ singleton prior and posterior distributions), one recovers an EB inequality for $[0,1]$-bounded i.i.d.\ random variables, with a larger constant on the $n^{-1/2}$ term compared to \eqref{eqn:mp-eb}, therefore is not sharp by our standard. An alternative to EB inequalities called the ``un-expected Bernstein inequality'' was introduced by \cite{mhammedi2019pac}, which ``together
with the standard Bernstein inequality imply [\emph{sic}] a version of the empirical Bernstein inequality with slightly worse factors'' \citep[Appendix G]{mhammedi2019pac}.
A much tighter, time-uniform PAC-Bayes EB inequality was later proved by \citet[Corollary 4]{jang2023tighter}, much similar to the ``betting'' technique of \cite{waudby2020estimating}. However, no sharp fixed-time EB inequality was proved by \cite{jang2023tighter}. A similar PAC-Bayes time-uniform result was obtained by \citet[Theorem 6]{mhammedi2021risk}, also suffering from non-sharpness when instantiated fixed-time.

We finally quote a recent study by \cite{shekhar2023near} that compares various scalar EB inequalities via their first- and second-order expansions. They show that a class of betting-driven confidence intervals, including those proposed by \cite{waudby2020estimating} (other than \eqref{eqn:eb-informal}) and the recent universal portfolio-based bound by \cite{orabona-up}, have a minimax optimal rate that outperforms even the sharpness criterion \eqref{eqn:sharp} for scalar EB inequalities. However, these confidence sets are not in closed form, and it remains an open problem (even in the scalar case) if there are closed-form empirical inequalities that achieve so.

%% file: eb.bbl
\begin{thebibliography}{25}
\providecommand{\natexlab}[1]{#1}
\providecommand{\url}[1]{\texttt{#1}}
\expandafter\ifx\csname urlstyle\endcsname\relax
  \providecommand{\doi}[1]{doi: #1}\else
  \providecommand{\doi}{doi: \begingroup \urlstyle{rm}\Url}\fi

\bibitem[Alquier(2024)]{alquier2024user}
P.~Alquier.
\newblock User-friendly introduction to {PAC-Bayes} bounds.
\newblock \emph{Foundations and Trends in Machine Learning}, 17\penalty0 (2):\penalty0 174--303, 2024.
\newblock ISSN 1935-8237.
\newblock \doi{10.1561/2200000100}.

\bibitem[Audibert et~al.(2009)Audibert, Munos, and Szepesv{\'a}ri]{audibert2009exploration}
J.-Y. Audibert, R.~Munos, and C.~Szepesv{\'a}ri.
\newblock Exploration--exploitation tradeoff using variance estimates in multi-armed bandits.
\newblock \emph{Theoretical Computer Science}, 410\penalty0 (19):\penalty0 1876--1902, 2009.

\bibitem[Chugg et~al.(2023)Chugg, Wang, and Ramdas]{chugg2023time}
B.~Chugg, H.~Wang, and A.~Ramdas.
\newblock Time-uniform confidence spheres for means of random vectors.
\newblock \emph{arXiv preprint arXiv:2311.08168}, 2023.

\bibitem[Fan et~al.(2015)Fan, Grama, and Liu]{fan2015exponential}
X.~Fan, I.~Grama, and Q.~Liu.
\newblock {Exponential inequalities for martingales with applications}.
\newblock \emph{Electronic Journal of Probability}, 20:\penalty0 1 -- 22, 2015.
\newblock \doi{10.1214/EJP.v20-3496}.

\bibitem[Hall and Heyde(2014)]{hall2014martingale}
P.~Hall and C.~C. Heyde.
\newblock \emph{Martingale Limit Theory and its Application}.
\newblock Academic press, 2014.

\bibitem[Hoeffding(1963)]{hoeffding1963probability}
W.~Hoeffding.
\newblock Probability inequalities for sums of bounded random variables.
\newblock \emph{Journal of the American Statistical Association}, 58\penalty0 (301):\penalty0 13--30, 1963.

\bibitem[Howard et~al.(2020)Howard, Ramdas, McAuliffe, and Sekhon]{howard2020time}
S.~R. Howard, A.~Ramdas, J.~McAuliffe, and J.~Sekhon.
\newblock Time-uniform {Chernoff} bounds via nonnegative supermartingales.
\newblock \emph{Probability Surveys}, 17:\penalty0 257--317, 2020.

\bibitem[Howard et~al.(2021)Howard, Ramdas, McAuliffe, and Sekhon]{howard2021time}
S.~R. Howard, A.~Ramdas, J.~McAuliffe, and J.~Sekhon.
\newblock Time-uniform, nonparametric, nonasymptotic confidence sequences.
\newblock \emph{The Annals of Statistics}, 49\penalty0 (2):\penalty0 1055--1080, 2021.

\bibitem[Jang et~al.(2023)Jang, Jun, Kuzborskij, and Orabona]{jang2023tighter}
K.~Jang, K.-S. Jun, I.~Kuzborskij, and F.~Orabona.
\newblock Tighter pac-bayes bounds through coin-betting.
\newblock In G.~Neu and L.~Rosasco, editors, \emph{Proceedings of Thirty Sixth Conference on Learning Theory}, volume 195 of \emph{Proceedings of Machine Learning Research}, pages 2240--2264. PMLR, 12--15 Jul 2023.

\bibitem[Kroshnin and Suvorikova(2024)]{kroshnin2024bernstein}
A.~Kroshnin and A.~Suvorikova.
\newblock Bernstein-type and bennett-type inequalities for unbounded matrix martingales.
\newblock \emph{arXiv preprint arXiv:2411.07878}, 2024.

\bibitem[Lieb(1973)]{lieb1973convex}
E.~H. Lieb.
\newblock Convex trace functions and the {W}igner-{Y}anase-{D}yson conjecture.
\newblock \emph{Les rencontres physiciens-math{\'e}maticiens de Strasbourg-RCP25}, 19:\penalty0 0--35, 1973.

\bibitem[Mackey et~al.(2014)Mackey, Jordan, Chen, Farrell, and Tropp]{mackey2014matrix}
L.~Mackey, M.~I. Jordan, R.~Y. Chen, B.~Farrell, and J.~A. Tropp.
\newblock {Matrix concentration inequalities via the method of exchangeable pairs}.
\newblock \emph{The Annals of Probability}, 42\penalty0 (3):\penalty0 906 -- 945, 2014.
\newblock \doi{10.1214/13-AOP892}.

\bibitem[Martinez-Taboada and Ramdas(2024)]{martinez2024empirical}
D.~Martinez-Taboada and A.~Ramdas.
\newblock Empirical {B}ernstein in smooth {B}anach spaces.
\newblock \emph{arXiv preprint arXiv:2409.06060}, 2024.

\bibitem[Maurer and Pontil(2009)]{maurer2009empirical}
A.~Maurer and M.~Pontil.
\newblock Empirical {Bernstein} bounds and sample variance penalization.
\newblock In \emph{22nd Annual Conference on Learning Theory}, 2009.

\bibitem[Mhammedi(2021)]{mhammedi2021risk}
Z.~Mhammedi.
\newblock Risk-monotonicity in statistical learning.
\newblock In \emph{Proceedings of the 35th International Conference on Neural Information Processing Systems}, NIPS '21, Red Hook, NY, USA, 2021. Curran Associates Inc.
\newblock ISBN 9781713845393.

\bibitem[Mhammedi et~al.(2019)Mhammedi, Gr{\"u}nwald, and Guedj]{mhammedi2019pac}
Z.~Mhammedi, P.~Gr{\"u}nwald, and B.~Guedj.
\newblock {PAC-Bayes} un-expected {Bernstein} inequality.
\newblock \emph{Advances in Neural Information Processing Systems}, 32, 2019.

\bibitem[Minsker(2017)]{MINSKER2017111}
S.~Minsker.
\newblock On some extensions of {B}ernstein's inequality for self-adjoint operators.
\newblock \emph{Statistics \& Probability Letters}, 127:\penalty0 111--119, 2017.
\newblock ISSN 0167-7152.
\newblock \doi{10.1016/j.spl.2017.03.020}.

\bibitem[Orabona and Jun(2024)]{orabona-up}
F.~Orabona and K.-S. Jun.
\newblock Tight concentrations and confidence sequences from the regret of universal portfolio.
\newblock \emph{IEEE Transactions on Information Theory}, 70\penalty0 (1):\penalty0 436--455, 2024.
\newblock \doi{10.1109/TIT.2023.3330187}.

\bibitem[Paulin et~al.(2016)Paulin, Mackey, and Tropp]{paulin2016efron}
D.~Paulin, L.~Mackey, and J.~A. Tropp.
\newblock {Efron–Stein inequalities for random matrices}.
\newblock \emph{The Annals of Probability}, 44\penalty0 (5):\penalty0 3431 -- 3473, 2016.
\newblock \doi{10.1214/15-AOP1054}.

\bibitem[Ramdas and Manole(2024)]{ramdas2023randomized}
A.~Ramdas and T.~Manole.
\newblock Randomized and exchangeable improvements of {M}arkov's, {C}hebyshev's and {C}hernoff's inequalities.
\newblock \emph{Statistical Science (to appear)}, 2024.

\bibitem[Shekhar and Ramdas(2023)]{shekhar2023near}
S.~Shekhar and A.~Ramdas.
\newblock On the near-optimality of betting confidence sets for bounded means.
\newblock \emph{arXiv preprint arXiv:2310.01547}, 2023.

\bibitem[Tolstikhin and Seldin(2013)]{tolstikhin2013pac}
I.~O. Tolstikhin and Y.~Seldin.
\newblock {PAC-B}ayes-empirical-{B}ernstein inequality.
\newblock In C.~Burges, L.~Bottou, M.~Welling, Z.~Ghahramani, and K.~Weinberger, editors, \emph{Advances in Neural Information Processing Systems}, volume~26. Curran Associates, Inc., 2013.

\bibitem[Tropp(2012)]{tropp2012user}
J.~A. Tropp.
\newblock User-friendly tail bounds for sums of random matrices.
\newblock \emph{Foundations of Computational Mathematics}, 12:\penalty0 389--434, 2012.

\bibitem[Tropp(2015)]{tropp2015introduction}
J.~A. Tropp.
\newblock An introduction to matrix concentration inequalities.
\newblock \emph{Foundations and Trends in Machine Learning}, 8\penalty0 (1-2):\penalty0 1--230, 2015.

\bibitem[Waudby-Smith and Ramdas(2023)]{waudby2020estimating}
I.~Waudby-Smith and A.~Ramdas.
\newblock Estimating means of bounded random variables by betting.
\newblock \emph{Journal of the Royal Statistical Society Series B: Statistical Methodology}, 86\penalty0 (1):\penalty0 1--27, 02 2023.
\newblock ISSN 1369-7412.
\newblock \doi{10.1093/jrsssb/qkad009}.

\end{thebibliography}
